\documentclass[11pt]{amsart}
\usepackage[pdftex]{graphicx,color}
\usepackage{fullpage}
\usepackage{boldfonts}
\usepackage{makecell}
\usepackage{comment}
\usepackage{enumerate}
\graphicspath{{/}{figs/}{Results/}}
\def\users{world} 
\def\R{\mathbb{R}}

\def\Q{\mathbb{Q}}

\def\cA{\mathcal{A}}

\def\cE{\mathcal{E}}
\def\cF{\mathcal{F}}

\def\cI{\mathcal{I}}

\def\cN{\mathcal{N}}

\def\cT{\mathcal{T}}

\def\d{\delta}

\def\p{\partial}
\def\o{\omega}
\def\veps{\varepsilon}

\def\O{\Omega}

\def\G{\Gamma}

\def\DD{{\rm D}}

\def\wto{\rightharpoonup}
\def\Id{I}
\def\transp{{\rotatebox[origin=c]{180}{\footnotesize $\perp$}}}

\newcommand{\dv}[1]{\,{\rm d}#1}

\newcommand{\wcheck}[1]{#1\hspace{-.8ex}\mbox{\huge {\lower.45ex \hbox{$\textstyle \check{}$}}} \hspace{.5ex}}

\DeclareMathOperator{\conv}{conv}
\DeclareMathOperator{\dist}{dist}
\DeclareMathOperator{\diam}{diam}

\let\oldmarginpar\marginpar
\renewcommand\marginpar[1]{
  \oldmarginpar[\raggedleft\footnotesize #1]
  {\raggedright\footnotesize #1}}
\usepackage{ifthen}
\usepackage[normalem]{ulem}
\ifthenelse{\equal{\users}{world}}{
  
  }
{

}
\newtheorem{definition}{Definition}

\newtheorem{proposition}{Proposition}
\newtheorem{theorem}{Theorem}
\newtheorem{corollary}{Corollary}
\newtheorem{remark}{Remark}

\theoremstyle{plain}

\newtheorem{algorithm}{Algorithm}
\numberwithin{definition}{section}
\numberwithin{lemma}{section}
\numberwithin{proposition}{section}
\numberwithin{theorem}{section}
\numberwithin{corollary}{section}
\numberwithin{equation}{section}
\numberwithin{remark}{section}
\numberwithin{remarks}{section}
\numberwithin{example}{section}
\numberwithin{examples}{section}
\definecolor{tourquoise}{RGB}{0,170,180}	
\definecolor{darkred}{RGB}{238,34,34}		
\definecolor{darkgreen}{RGB}{0,190,0}		
\definecolor{lightgray}{RGB}{210,210,210}	
\definecolor{deepblue}{RGB}{0,0,240}		
\definecolor{darkgray}{RGB}{144,144,144}	
\definecolor{kingblue}{RGB}{64,96,224}		
\definecolor{gold}{RGB}{240,208,0}		
\definecolor{verydarkred}{RGB}{176,0,0}		

\def\be{{\bf e}}
\def\tE{\widetilde{E}}
\def\tPhi{\widetilde{\Phi}}

\def\G{\mathbb{G}}
\def\bnu{{\bf \nu}} 
\def\tby{\widetilde{\by}}
\def\W{\mathbb{W}}
\def\second{H}
\def\tr{\textrm{tr }}
\def\hcI{{\cI}}

\newcommand{\cpam}[1]{\textcolor{black}{#1}}

\begin{document}
\begin{sloppypar}
\title[Bilayer Plate Modeling and Numerics]{Bilayer Plates: Model
  Reduction, $\Gamma$-Convergent Finite Element Approximation
  and Discrete Gradient Flow} 
\author{S\"oren Bartels}
 \address{Albert-Ludwigs-Universit\"at Freiburg, Germany.}
\email{bartels@mathematik.uni-freiburg.de}
\author{Andrea Bonito$^\dag$}
\address{Texas A\&M University, College Station, TX.}
\email{bonito@math.tamu.edu}
\thanks{\mbox{}$^\dag$ Partially supported by NSF Grant DMS-1254618 and AFOSR Grant FA9550-14-1-0234.}
\author{Ricardo H. Nochetto$^*$}
\address{University of Maryland, College Park, MD.}
\email{rhn@math.umd.edu}
\thanks{\mbox{}$^*$ Partially supported by NSF Grants DMS-1109325 and DMS-1411808}
\date{\today}
\begin{abstract}
The bending of bilayer plates is a mechanism which allows for large deformations
via small externally induced lattice mismatches of the underlying materials. 
Its mathematical modeling, discussed herein, consists of 
a nonlinear fourth order problem with a pointwise isometry
constraint. A discretization based on Kirchhoff quadrilaterals is devised 
and its $\Gamma$-convergence is proved. An iterative method that
decreases the energy is proposed and its convergence to stationary
configurations is investigated.
Its performance, as well as reduced model capabilities, are
explored via several insightful numerical experiments involving large 
(geometrically nonlinear) deformations.
\end{abstract}

\maketitle

\section{Introduction}

The derivation of dimensionally reduced mathematical models
and their numerical treatment is a classical and challenging scientific branch 
within solid mechanics. Various models for describing the bending or membrane behavior 
of plates are available, either as linear models for the description of small displacements 
or nonlinear models when large deformations are considered; see~\cite{Ciar:97}. 
The development of
related numerical methods has mostly been concerned with the treatment of second order
derivatives and avoiding various locking effects. 
The rigorous derivation of the geometrically 
nonlinear Kirchhoff model for the description
of large bending deformations of plates from three-dimensional hyperelasticity
in~\cite{FrJaMu:02} has inspired various further results, e.g., discussing other 
energy regimes in~\cite{FrJaMu:06}, or the derivation of effective theories for 
prestressed multilayer materials in~\cite{Schm:07b}.

Bilayer plates consist of two films of different materials glued on top of each other. 
The materials react differently to thermal or electric
stimuli, thereby changing their molecular lattices. This mismatch 
allows for the development of large deformations by simple 
heat or electric actuation. Classical applications of this effect
include bimetal strips in thermostats, while modern applications use thermally 
and electrically induced bending effect to produce nanorolls,
microgrippers, and nano-tubes;
see~\cite{BALG:10,JSI:00,KLPL:05,SchEb:01,SIL:95}. 
Preventing undesirable effects such
as {\em dog-ears} formation, as reported in~\cite{AlBaSm:11,ADMA:ADMA19930050905},
motivates the mathematical prediction of bilayer bending patterns
via numerical simulation. This requires having a model as simple
  as possible to be ameanable to numerical treatment and analysis,
  but sufficiently sophisticated to capture essential nonlinear
  geometric features associated with large bending deformations.

A two-dimensional mathematical model for the bending behavior of bilayers 
has been rigorously derived from three-dimensional hyperelasticity in~\cite{Schm:07b}.
It consists of a nonconvex minimization problem with nonlinear pointwise constraint. 
The energy functional involves second order derivatives of deformations 
associated with the second fundamental form of the mid-surface.
The pointwise constraint enforces deformations
to be {\it isometries}, i.e., that length and angle relations remain unchanged by
the deformation as in the case of the bending of a piece of paper. A related
numerical method has been devised and analyzed for single layer plates
in~\cite{Bart:13,Bart15-book}. 

It is our goal to develop a reliable numerical method for the practical 
computation of large bilayer bending deformations. Our
  contributions in this paper are:

\begin{enumerate}[$\bullet$]
\item
  To present a {formal} dimension reduction model allowing
  for various effects not covered \cpam{in the corresponding
  rigorous analysis} in~\cite{Schm:07a};

\item
To propose a discretization of the mathematical model and
prove its $\Gamma$-convergence as discretization parameters tend to
zero;

\item 
To construct a gradient flow method to compute stationary
configurations;

\item
To carry out several numerical experiments
to illustrate the performance of our numerical method and explore
the nonlinear geometric effects captured by
the mathematical model. 
\end{enumerate}

In the remainder of this introduction we 
discuss the mathematical model and our main ideas to deal with the
ensuing strong nonlinearities. 

\medskip
{\em Description of bilayer plates}. We consider a geometrically nonlinear 
Kirchhoff plate model that allows for bending but not for stretching or shear. 
This selection is related to the choice of a particular energy scaling, namely
that the elastic energy is proportional to the third power of the plate
thickness $t$.  
Given a domain $\omega\subset\mathbb{R}^2$ that describes the middle
surface of the plate, the model \cpam{formally derived in Section~\ref{S:model-reduction}}
consists of minimizing the 
dimensionally reduced elastic energy
\begin{equation}\label{energy-1}
E[\by] = \frac{1}{2} \int_\o \big| \second + Z \big|^2 - \int_\o \bef \cdot \by
\end{equation}
within the set of isometries $\by:\omega\to\mathbb{R}^3$, i.e., 
mappings satisfying the identities
\begin{equation}\label{isometry}
[\nabla \by]^T\nabla\by = I_2 \qquad \Longleftrightarrow \qquad 
\p_i \by \cdot \p_j \by = \d_{ij}, \quad i,j=1,2, 
\end{equation}
in $\o$ and with prescribed values $\by=\by_D$ and $\nabla\by=\Phi_D$ on the
Dirichlet portion $\partial_D\omega$ of the boundary $\partial\omega$.
Hereafter, $I_d$ denotes the identity matrix in $\R^d$ for $d=2,3$, 
and $H$ stands for the second fundamental form of the surface
$\gamma = \by(\omega)$ parametrized by $\by$ with unit normal $\bnu$,
\[
H_{i,j} = \bnu \cdot \p_i\p_j \by, \quad 
\bnu = \partial_1\by \times \partial_2\by.
\]
The symmetric matrix $Z$ is given and can be
viewed as a spontaneous curvature so that in the absence of body
forces $\bef$ the plate is already pre-stressed.
Given the identity for isometries
\[
|\second|^2 = |D^2\by|^2,
\]
we rewrite the energy $E[\by]$ in \eqref{energy-1} as follows:
\begin{equation}\label{e:energy-used}
\tE[\by] :=\frac12 \int_\o |D^2 \by|^2  
+ \sum_{i,j=1}^2 \int_\o \p_i\p_j \by \cdot 
\Big(\frac{\p_1 \by}{|\p_1 \by|} \times \frac{\p_2 \by}{|\p_2 \by|}\Big) Z_{ij}
+ \frac12 \int_\o |Z|^2  - \int_\o \bef\cdot \by .
\end{equation}
It is tempting to simplify this expression further because 
$| \p_i \by | =1$, $i=1,2$, for isometries.
However, we refrain from doing so for stability purposes anticipating
that the isometry constraint will be later relaxed numerically.
\cpam{In particular, the normalization will enable us to prove various bounds
for the variational derivative of the energy.}
In fact, notice that the energy $\tE[\by]$ is finite for
$\by \in H^2(\omega)^3\cap W^1_\infty(\omega)^3$ such that $|\p_i \by
|\geq 1$, $i=1,2$, as well as $\bef \in L^2(\omega)^3$ and $Z \in
  L^2(\omega)^{2\times 2}$. The condition $|\p_i \by |\geq 1$, $i=1,2$,
will play a crucial role throughout this paper. We encode boundary
conditions and the isometry constraint in the set of 
admissible deformations
\begin{equation}\label{admin-set}
\mathcal{A} :=  \big\{\by\in H^2(\omega)^3:
\by|_{\partial_D\omega} = \by_D, \nabla\by|_{\partial_D\omega} = \Phi_D,
[\nabla\by]^T\nabla\by=I_2 ~\text{a.e. in }\o\big\}
\end{equation}
and define the tangent space of $\cA$ at a point $\by\in \cA$ via
\begin{equation}\label{admin_tangent}
\cF[\by] := \big\{\bw \in H^2(\omega)^3:
\bw|_{\partial_D\omega} = 0, \nabla\bw|_{\partial_D\omega} = 0,\ 
[\nabla\bw]^T \nabla \by + [\nabla \by]^T \nabla \bw =0 ~\text{a.e. in }\o
\big\}.
\end{equation}
{
Note that it is always possible to extend the functional $\tE$ to $H^1(\o)^3$ as follows:
$$
\tE[\by] := +\infty, \qquad \by \in H^1(\o)^3 \setminus \cA.
$$
}

{\em Minimizing movements.} To find stationary points of $\tE$ in $\mathcal{A}$,
we propose a gradient flow in $H^2(\omega)$, i.e. according to the $H^2$-scalar product:
if {$s\in(0,\infty)$} is a pseudo-time, we formally seek a family 
$\by \in L^2(0,\infty; \mathcal A)$ with {$\partial_s\by(s)  \in \cF[\by(s)]$ 
for $s\in (0,\infty)$ and
\[
\int_\o D^2 \big(\partial_s\by(s)\big) : D^2 \bw = - \delta\tE[\by(s),\bw] 
\qquad \forall \bw \in \cF[\by(s)].
\]
The expression $\delta \tE[\by,\bw]$ stands for the variational
derivative of $\tE$ at $\by \in \mathcal A$ in the direction 
$\bw \in \mathcal F[\by(s)]$, which we make explicit below. We note that upon
taking $\bw=\partial_s\by(s)\in\mathcal{F}[\by(s)]$, we obtain formally
\[
\frac{d}{ds} \tE[\by(s)] = 
- \int_\omega \big|D^2\big(\partial_s\by(s)\big)\big|^2 \le 0;
\]
the energy thus decreases along trajectories.} The formal gradient flow is
highly \cpam{nonlinear and requires} an appropriate interpretation. We adopt 
the concept of {\it minimizing movements} and consider 
an implicit first order time-discretization of the $H^2$-gradient flow
via successive minimization of
\[
\by \mapsto \frac{1}{2\tau} \|D^2 (\by-\by^k)\|_{L^2(\omega)}^2 + \tE[\by]
\]
in the set of all $\by \in \cA$ {to determine $\by^{k+1}$.}
\cpam{Our motivation for the use of
the $H^2$ metric to define the gradient flow is threefold.
First, it simplifies the implementation
since it leads to the same system matrix as the main part of the bilinear form associated
with the bending energy \eqref{e:energy-used}. Second, it is
sufficiently strong to provide control over discrete time
derivatives and discretization errors, such as the isometry constraint,
without imposing severe step size restrictions. Third,
it may be regarded as a damping term modeling the deceleration of
a bilayer within a viscous fluid, which in turn gives some physical interpretation.}
Since the nonlinear isometry constraint
is treated exactly, this nonconvex minimization problem is of limited
practical value. We thus propose instead a linearization of the
isometry constraint which yields a practical scheme. 

\begin{algorithm}[minimizing movement]\label{alg:grad_flow_cont}
Let $\tau >0$ be the time-step size and set $k=0$. Choose $\by^0 \in \mathcal A$.
Compute $\bv^{k+1} \in \mathcal F[\by^k]$ which is minimal for the functional
\[
\bv \mapsto \frac{\tau}{2} \|D^2 \bv \|_{L^2(\omega)}^2 + \tE[\by^k+\tau \bv],
\]
set $\by^{k+1}= \by^k + \tau \bv^{k+1}$, increase $k\to k+1$ and repeat.
\end{algorithm}
The linearized isometry condition included in the set  
$\mathcal{F}[\by^k]$ implies that for 
every admissible vector field $\bv \in \mathcal F[\by^k]$ we have that
the constraint residual of the corresponding update $\by^k + \tau \bv$ satisfies 
\[\begin{split}
[\nabla (\by^k+\tau \bv)]^\transp [\nabla (\by^k+\tau \bv)] - \Id_2 
&= [\nabla \by^k]^\transp [\nabla \by^k] -\Id_2
+ \tau^2 [\nabla \bv ]^\transp [\nabla \bv]  \geq [\nabla \by^k]^\transp [\nabla \by^k] -\Id_2,
\end{split}\]
where the inequality $A \ge B$ for square symmetric matrices $A,B$ means
  that $A-B$ is semi-positive definite; in particular the diagonal
  entries of $A$ and $B$ satisfy $a_{ii}\ge b_{ii}$ for all $i$.
Applying this formula inductively with $[\nabla \by^0]^\transp [\nabla \by^0] =\Id_2$,
we see that $[\nabla \by^{k+1}]^\transp \nabla \by^{k+1} \geq I_2$ whence $|\p_j \by^{k+1}|\ge 1$, $j=1,2$.
Therefore, $\tE[\by^{k+1}]$ is well defined and the minimization problem in
Algorithm \ref{alg:grad_flow_cont} admits a minimizer. Moreover, the 
space $\cF[\by^{k+1}]$ is well defined, even though $\by^{k+1}$ may not belong
to $\cA$. Hence, the iteration of Algorithm~\ref{alg:grad_flow_cont} can be
repeated.
Every iteration of the algorithm requires computing
a solution $\bv^{k+1} \in  \mathcal F[\by^k]$ of the following
Euler--Lagrange equation
\[
\tau \int_\omega D^2 \bv^{k+1} : D^2\bw + \delta \tE[\by^{k}+\tau \bv^{k+1},\bw] = 0
\quad\forall \bw\in\mathcal{F}[\by^k].
\]
In terms of the new iterate $\by^{k+1} = \by^k + \tau \bv^{k+1}$ this is equivalent to 
the nonlinear system of equations 
\begin{equation}\label{min-mov}
\begin{aligned}
\frac{1}{\tau} \int_\o & D^2 (\by^{k+1} - \by^k) : D^2\bw 
+ \int_\o D^2 \by^{k+1}:D^2\bw  
\\
+& \sum_{i,j=1}^2 \int_\o  \p_i\p_j \bw\cdot  
\Big(\frac{\p_1 \by^{k+1}}{|\p_1 \by^{k+1}|} \times \frac{\p_2 \by^{k+1}}{|\p_2 \by^{k+1}|}\Big) Z_{ij}
\\
+ &\sum_{i,j=1}^2\int_\o \p_i\p_j \by^{k+1}\cdot
\Big\{ \Big[\frac{\p_1 \bw }{|\p_1 \by^{k+1}|}-\frac{\p_1
  \by^{k+1}(\p_1\by^{k+1}\cdot\p_1\bw)}{|\p_1 \by|^3}\Big]
 \times \frac{\p_2 \by^{k+1}}{|\p_2 \by^{k+1}|} \Big\} Z_{ij} \\
+  &\sum_{i,j=1}^2\int_\o \p_i\p_j \by^{k+1}\cdot\Big\{\frac{\p_1\by^{k+1}}{|\p_1\by^{k+1}|}
 \times\Big[\frac{\p_2 \bw }{|\p_2 \by^{k+1}|}-\frac{\p_2 \by^{k+1}(\p_2 \by^{k+1}\cdot\p_2 \bw)}{|\p_2 \by^{k+1}|^3}\Big]\Big\} Z_{ij}
= \int_\o \bef \cdot \bw ,
\end{aligned}
\end{equation}
for all $\bw \in \mathcal \cF[\by^k]$.
We show in this paper how to discretize this system in space and
present an iterative algorithm for its approximation. We study these
algorithms and employ them to compute several equilibrium configurations.

\medskip
{\em Outline of the paper.}
The paper is organized as follows. In \S \ref{S:model-reduction}
we discuss a formal derivation of \eqref{energy-1} from three
dimensional hyperelasticity, following a suggestion of S. Conti.
In \S \ref{S:Kirchhoff} we introduce {\it Kirchhoff quadrilaterals}, 
which is a nonconforming finite element specially taylored for this
application. It is an extension of the Kirchhoff triangles
{\cite{Brae:07,Bart:13},} and possesses the degrees of freedom for the
function and its gradient at the vertices of the underlying partition
$\mathcal{T}_h$; \cpam{this facilitates imposing the isometry
constraint at the vertices}. It turns out that the space of deformations as well
as that of discrete gradients are subspaces of $H^1(\omega)$, which is
extremely convenient to discretize \eqref{min-mov}.
{
We next introduce space discretizations $\tE_h$ of $\tE$ 
and derive in \S \ref{S:convergence} their $\Gamma$-convergence to
$\tE$ in $H^1(\o)$.
As a consequence, we deduce convergence properties of discrete almost absolute minimizers.
Then, we study a practical iterative algorithm for the solution of  the space discretization of \eqref{min-mov}. 
We show convergence of such an iterative scheme, thereby
proving existence and uniqueness of our fully discrete problem. We
also prove several important properties of this gradient flow, such as
a precise control of the deviation from the isometry constraint \eqref{isometry}.}
We conclude in \S \ref{S:experiments}
with insightful numerical experiments. In fact,
we compute several configurations (such as cylinders, dog ears, corkscrews)
that are observed in experiments with \cpam{micro- and
nano-devices~\cite{AlBaSm:11,ADMA:ADMA19930050905}}.
\cpam{We emphasize that some of these effects, e.g. corkscrew shapes,
  are obtained with anisotropic spontaneous curvatures $Z$ and are therefore outside the framework developed and analyzed in \cite{Schm:07a}.}

\section{Dimension Reduction: Bilayer Plate Model}\label{S:model-reduction}
We consider a plate $\o_t := \o\times (-t/2,t/2) \subset \R^3$ 
of thickness $t>0$ and whose middle surface is
given by $\o\subset \R^2$ as illustrated in \cpam{Figure
\ref{f:model-reduction-setting}} (left).
The plate is clamped on the left edge $\partial_D\o$ and free on the rest of the
\cpam{boundary, and} its length perpendicular to $\partial_D\o$ is $L$.
The upper and lower layers are composed of materials with different
molecular lattices, for instance differing by a factor $\delta_t > 0$;
this could be achieved by thermal or electric actuation in practice.
To understand the natural scaling between $t$ and $\delta_t$, we
assume that the \cpam{middle surface of the} upper layer contracts to length
$L(1-\delta_t)$ whereas the \cpam{middle surface of the} lower layer expands to length
$L(1+\delta_t)$, so that the plate $\omega_t$ bends upwards as in 
\cpam{Figure \ref{f:model-reduction-setting}. Due to the clamped
boundary condition along one side} we imagine that for small
deformations the \cpam{lower and upper middle surfaces are given by 
$R_{\pm}=\theta^{-1} L(1 \pm \delta_t)$}, as depicted in \cpam{Figure
\ref{f:model-reduction-setting}} (right). Since we aim at capturing
bending effects in the limit $t\to0$, we impose the condition
\[
\lim_{t\to0} \frac{1}{R_+} = \lim_{t\to0} \frac{1}{R_-} =
\frac{\theta}{L} = \kappa,
\]
$\kappa>0$ being the curvature. Since 
\cpam{$t=R_+-R_- = \kappa^{-1}\big((1+\delta_t)-(1-\delta_t)\big)$,} we deduce
\begin{equation}\label{t-delta}
\lim_{t\to0} \frac{t}{\delta_t} = \frac{2}{\kappa}.
\end{equation}
It is thus natural to impose a linear scaling between $t$ and
$\delta_t$ which involves the curvature that is expected in the limit
of vanishing thickness for a pure bending problem.
\begin{figure}[ht!]
\begin{picture}(0,0)%
\includegraphics{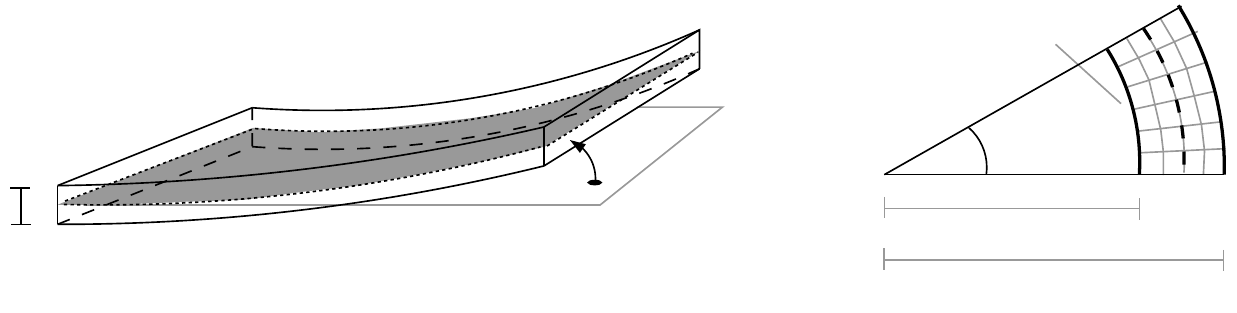}%
\end{picture}%
\setlength{\unitlength}{4144sp}%
\begingroup\makeatletter\ifx\SetFigFont\undefined%
\gdef\SetFigFont#1#2#3#4#5{%
  \reset@font\fontsize{#1}{#2pt}%
  \fontfamily{#3}\fontseries{#4}\fontshape{#5}%
  \selectfont}%
\fi\endgroup%
\begin{picture}(6246,1415)(41,-638)
\put(4859,-578){\makebox(0,0)[b]{\smash{{\SetFigFont{10}{12.0}{\familydefault}{\mddefault}{\updefault}{\color[rgb]{0,0,0}$R_+=R_-+t$}%
}}}}
\put(4404, 41){\makebox(0,0)[b]{\smash{{\SetFigFont{10}{12.0}{\familydefault}{\mddefault}{\updefault}{\color[rgb]{0,0,0}$\theta$}%
}}}}
\put(2804, 39){\makebox(0,0)[lb]{\smash{{\SetFigFont{10}{12.0}{\familydefault}{\mddefault}{\updefault}{\color[rgb]{0,0,0}$\by$}%
}}}}
\put(3188, 17){\makebox(0,0)[lb]{\smash{{\SetFigFont{10}{12.0}{\familydefault}{\mddefault}{\updefault}{\color[rgb]{0,0,0}$\o$}%
}}}}
\put(2902,-233){\makebox(0,0)[lb]{\smash{{\SetFigFont{10}{12.0}{\familydefault}{\mddefault}{\updefault}{\color[rgb]{0,0,0}$\bx'\in\o$}%
}}}}
\put(4635,-320){\makebox(0,0)[b]{\smash{{\SetFigFont{10}{12.0}{\familydefault}{\mddefault}{\updefault}{\color[rgb]{0,0,0}$R_-$}%
}}}}
\put(4818,564){\makebox(0,0)[rb]{\smash{{\SetFigFont{10}{12.0}{\familydefault}{\mddefault}{\updefault}{\color[rgb]{0,0,0}$\theta R_- = L (1-\d_t)$}%
}}}}
\put(5671,389){\makebox(0,0)[lb]{\smash{{\SetFigFont{10}{12.0}{\familydefault}{\mddefault}{\updefault}{\color[rgb]{0,0,0}$L(1+\delta_t)=\theta R_+$}%
}}}}
\put( 56,-203){\makebox(0,0)[rb]{\smash{{\SetFigFont{10}{12.0}{\familydefault}{\mddefault}{\updefault}{\color[rgb]{0,0,0}$t$}%
}}}}
\put(1561,-530){\makebox(0,0)[b]{\smash{{\SetFigFont{10}{12.0}{\familydefault}{\mddefault}{\updefault}{\color[rgb]{0,0,0}$L$}%
}}}}
\put(601,299){\makebox(0,0)[rb]{\smash{{\SetFigFont{10}{12.0}{\familydefault}{\mddefault}{\updefault}{\color[rgb]{0,0,0}$\partial_D\o$}%
}}}}
\end{picture}%
\caption{\label{f:model-reduction-setting} Two layers of thickness $t/2$ are stacked on
each other and form the bilayer plate. The undeformed {middle} surface 
is denoted by $\omega$ and
deforms into the surface~$\gamma$. \looseness-1}
\end{figure}
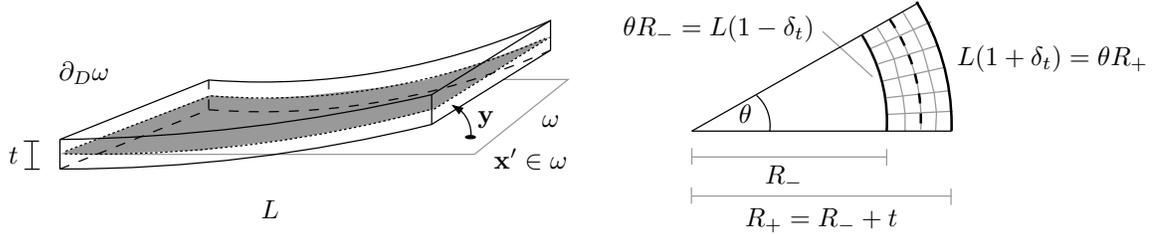

We consider the energy density $W:\R^{3\times 3} \times \o_t \to \R$
\begin{equation}\label{energy-density}
W(F,\bx) = \frac14 \Big|F^\transp F - (I_3\pm \d_t N(\bx'))^\transp 
(I_3\pm \d_t N(\bx'))\Big|^2
\quad\pm x_3>0,
\end{equation}
where $\bx=(\bx',x_3)$, $I_3$ denotes the identity matrix
in $\mathbb R^{3\times 3}$, and $|A|$ stands for the norm
associated to the \cpam{Frobenius} scalar product $A:B:=\sum_{i,j=1}^d
A_{ij}B_{ij} = \tr(A^T B)$.
Hereafter $\d_t$ is a parameter only depending on the thickness $t$,
describing the lattice mismatch of the two layers and
satisfying $\d_t \sim t$, whereas  
$N:\o \to \R^{3\times 3}$ is a symmetric matrix which encodes
inhomogeneities (dependence on $\bx'$) and anisotropy 
(rectangular molecular lattice rather than cubic and preferred
directions) of the underlying materials.
Together $\delta_t$ and $N(\bx')$ describe the pre-stressed
bilayer $\{\bx\in\o_t: \pm x_3 > 0\}$. When $\delta_t=0$ the two
materials composing the bilayers reduce to one, the reference
configuration is stationary in the absence of a force $\bef_t$ 
and thus stress-free, and the energy density becomes
\begin{equation}\label{simplest}
W(F) = \frac14 \Big| F^TF-I_3 \Big|^2.
\end{equation}
This function is asymptotically equivalent to the simplest energy density that obeys the 
principles of frame indifference and isotropy, namely $W(F)=W(QFR)$ for 
all $Q,R\in SO(3)$ and, see~\cite{FrJaMu:06},
\begin{equation}\label{frame-indiference}
W(F) \approx \dist^2\big(F,SO(3)\big),
\end{equation}
where $\dist$ is given by the \cpam{Frobenius} metric.
To see the relation between \eqref{simplest} and
\eqref{frame-indiference} we argue as follows. Let $F$ be close to
$SO(3)$, which is to say $F=F_0 + \epsilon F_1$ with $F_0 \in SO(3)$ and 
$F_1$ perpendicular to the tangent space $T_{F_0} SO(3)$ to $SO(3)$ at $F_0$ and
$\epsilon\ll 1$; we thus deduce $\dist^2(F,SO(3)) = \epsilon^2 |F_1|^2$. 
The space $T_{F_0} SO(3)$ can be written as
\[
T_{F_0} SO(3) = \big\{ Z: ~ F_0^T Z + Z^T F_0 = 0 \big\};
\]
this follows by differentiation of the condition $F_0^T F_0=I$.
Consequently, the normal space $N_{F_0} SO(3)$ to $T_{F_0} SO(3)$ is
\[
N_{F_0} SO(3) = \big\{Y: ~  F_0^T Y - Y^T F_0 = 0 \big\},
\]
as can be easily seen because $T_{F_0} SO(3) \oplus N_{F_0} SO(3) =
\mathbb{R}^{3\times 3}$ and
\[
Z : Y = \tr (Z^TY) = \tr\big( (Z^TF_0) (F_0^TY)\big) = - \tr \big( (Y^TF_0) (F_0^TZ)\big)
= - \tr (Y^TZ) = -Z:Y,
\]
whence $Z:Y=0$ and the subspaces are orthogonal.
Since $F_1\in N_{F_0} SO(3)$ we infer that
\[
\big| F^TF- I_3 \big|^2 = \big| (F_0 +\epsilon F_1)^T (F_0+\epsilon F_1) - I_3 \big|^2
= \big| 2\epsilon F_1^T F_0 + \epsilon^2 F_1^T F_1  \big|^2 = 4 \epsilon^2 |F_1|^2
+ o(\epsilon^2),
\]
which shows the asserted relation between \eqref{simplest} and
\eqref{frame-indiference} for small $\epsilon$.

\cpam{We are interested in the bending regime of the bilayer, which corresponds
to energies comparable to the third power of the plate thickness, 
cf.~\cite{FrJaMu:06,Schm:07b}.}
To a deformation $\bu:\o_t \to \R^3$ 
of the plate we \cpam{thus} associate the scaled hyperelastic energy 
\begin{equation}\label{cubic-scaling}
I_t[\bu] = t^{-3} \int_{\o_t} \Big( W(\nabla \bu,\bx) - \bef_t \cdot \bu
\Big) \dv{\bx}.
\end{equation}
The function $\bef_t$ is a body force, whereas
the energy density $W$ is written in \eqref{energy-density} and reads
\[
W(F,\bx) 
= \frac14 \big| F^\transp F - M \big|^2,
\]
with symmetric matrices $M=M(\bx),N=N(\bx)\in\mathbb{R}^{3\times3}$ given by
\[
M := \begin{bmatrix}
M_{11} & M_{12} \\ M_{12}^\transp & M_{22}
\end{bmatrix}
:= I_3 \pm  2 \d_t N  + \d_t^2 N^2,
\qquad
N \cpam{:=} \begin{bmatrix}
N_{11} & \bm \\ \bm^\transp & n
\end{bmatrix}.
\]
\cpam{Here and in what follows alternating signs correspond to the upper
and lower layers in which we have $x_3>0$ and $x_3<0$, respectively, i.e., 
$\pm x_3>0$.}
Moreover, $N_{11}\in \R^{2\times2}$ is symmetric, $\bm\in \R^2$,
$n\in\R$ is constant \cpam{(to keep the formal discussion simple)}, and
\[\begin{split}
M_{11} &= I_2 \pm 2\d_t N_{11} + \d_t^2 (N_{11}^2 + \bm \bm^T), \\
M_{12} &= \pm 2\d_t \bm + \d_t^2(N_{11} \bm + n \bm), \\
M_{22} &= 1 \pm 2\delta_t n + \delta_t^2(n^2+|\bm|^2).
\end{split}\]

To derive a dimensionally reduced model we assume that the actual
deformation $\bu$ of the plate, subject to boundary conditions and 
outer forces, has the form 
\[
\bu(\bx',x_3) = \by(\bx') + x_3 \bb(\bx')
\]
with a vector field $\bb:\o\to\R^3$ that is perpendicular to the surface 
$\gamma$ parametrized by $\by$, i.e., we have $\p_i \by \cdot \bb = 0$ for $i=1,2$.
In other words, fibers orthogonal to the middle surface in the reference
configuration remain normal to $\gamma$ and deform linearly. This
special form of $\bu$ is consistent with \cite{FrJaMu:02,Schm:07b} 
\cpam{for energy densities with a vanishing bulk modulus. In general, a 
more general expansion including quadratic terms in $x_3$ has to be used.}

With this ansatz we find that $\nabla \bu = [\partial_i \bu]_{i=1}^3 \in
\R^{3\times3}$ can be written as
\[
\nabla \bu  = [\nabla' \by, \bb] + x_3 [\nabla' \bb,0],
\]
where $\nabla'$ stands for the gradient with respect to $\bx'$, and deduce that
\begin{equation*}
\begin{aligned}
I_t[\bu] = &\frac{1}{t^3} \int_{\o_t} \Big(\frac14\big| (\nabla
\bu)^\transp \nabla \bu - M \big|^2 - \cpam{\bef_t}\cdot\bu \Big) \dv{\bx}
 = \frac{1}{t^3} \int_{\o_t} \left\{ \frac14 \left|
\begin{bmatrix}
(\nabla' \by)^\transp (\nabla' \by) - M_{11} & -M_{12} \\
-M_{12}^T & |\bb|^2 - M_{22}
\end{bmatrix} \right. \right. \\
\\
& + x_3
\begin{bmatrix}
 (\nabla'\bb)^\transp \nabla' \by + (\nabla' \by)^\transp \nabla' \bb & (\nabla' \bb)^\transp \bb \\ 
\bb^\transp (\nabla' \bb) & 0 
\end{bmatrix}
+ x_3^2 \left. \left.
\begin{bmatrix}
(\nabla'\bb)^\transp \nabla' \bb & 0 \\ 0 & 0 
\end{bmatrix}
\right|^2 - \bef_t\cdot\bu \right\} \dv{\bx}.
\end{aligned}
\end{equation*}

In order for this
integral to be bounded as $t\to 0$ we need that the term
$|b|^2 - M_{22}$
be at least of order $t^2$. Since we have $\d_t\sim t$, 
this is guaranteed if we enforce that 
\[
|\bb|^2 - (1 \pm 2 \d_t n) -\d_t^2(n^2 + |\bm|^2)
= |\bb|^2 - (1\pm \d_t n)^2 - \d_t^2 |\bm|^2 
= -\d_t^2 |\bm|^2,
\]
i.e., we impose the constraint
\[
|\bb|=1\pm \d_t n \qquad
\pm x_3>0.
\]
Since $\bb(\bx')=\beta(\bx') \bnu(\bx')$, where $\bnu(\bx'):=\frac{\p_1 \by(\bx') \times \p_2 \by(\bx')}{|\p_1 \by(\bx') \times \p_2 \by(\bx')|}$ is the unit normal to the surface $\gamma$ at $\by(\bx')$,
we obtain that $\beta(\bx') =1\pm \d_t n$ \cpam{which is for simplicity assumed to be}
independent of $\bx'\in\o$. This in turn implies that
$$
(\nabla' \bb)^\transp \bb = 0, \qquad \nabla' \bb = (1\pm \d_t n) \nabla' \bnu.
$$
Recalling that the first and second fundamental forms of $\gamma$
are given by
\begin{equation}\label{fundamental-forms}
G = (\nabla' \by)^\transp \nabla' \by,  \quad
\second = - (\nabla' \bnu)^\transp \nabla' \by, 
\end{equation}
and introducing
\begin{align*}
G_t &:= t^{-1} \big((\nabla' \by)^\transp(\nabla' \by) - M_{11}\big) =
t^{-1} \big(G - I_2 \mp 2 \d_t N_{11} - \d_t^2 (N_{11}^2 + \bm \bm^\transp) \big), \\
K_t &:= (\nabla' \bb)^\transp \nabla' \bb = (1\pm \d_t n)^2 (\nabla' \bnu)^\transp \nabla ' \bnu,
\end{align*}
we infer that
\begin{equation*}
I_t[\bu] = \frac{1}{t^3} \int_{\o_t} \left\{ \frac14 \left|
\begin{bmatrix}
tG_t -2x_3 (1\pm \d_t n)\second +x_3^2 K_t & -M_{12} \\
-M_{12}^T & -\d_t^2 |\bm|^2
\end{bmatrix}\right|^2 - \bef_t\cdot\bu \right\} \dv\bx.
\end{equation*}
Retaining only the terms of order \cpam{$\d_t^2$} or lower, because
the higher order terms vanish in the limit $t\to0$, we obtain
\[
\begin{split}
I_t[\bu] &\approx
\frac{1}{t^3} \int_{\o_t} 
\left\{\frac14\Big( t^2 |G_t|^2 + 4 x_3^2 (1\pm \d_t n)^2 |\second|^2 + x_3^4 |\cpam{K_t}|^2 \right.\\
& \qquad \left.
 -4 t x_3 (1\pm \d_t n) G_t : \second 
 + 2 t x_3^2 G_t : \cpam{K_t} 
 - 4 x_3^3 (1\pm \d_t n) \second:\cpam{K_t}  + 8 \d_t^2 |\bm|^2 \Big) -
 \bef_t\cdot\bu \right\} \dv\bx.
\end{split}
\]
To ensure that the first term in the integral remains bounded
as $t\to 0$ we require that 
\[
G=I_2,
\] 
that is the parametrization $\by$ of $\gamma$ is an {\it isometry}.
Consequently, we obtain
\[
G_t = \mp 2 t^{-1} \d_t N_{11} - t^{-1} \d_t^2 P,
\qquad
P := N_{11}^2 + \bm \bm^\transp.
\]
Since the quantities $N_{11}, P, \second, K, \bm$ are independent of $x_3$,
we carry out the integration over $x_3\in (-t/2,t/2)$ and deduce
that
\begin{equation*}
\begin{aligned}
\int_{-t/2}^{t/2} t^2 |G_t|^2 \dv x_3 & = 4t \d_t^2 |N_{11}|^2 + t\d_t^4 |P|^2,
\\
\int_{-t/2}^{t/2} 4 x_3^2 (1\pm \d_t n \cpam{)^2} |\second|^2 \dv x_3 & = \frac{t^3}{3}
(1+\d_t^2 n^2) |\second|^2,
\\
\int_{-t/2}^{t/2} x_3^4 |K|^2 \dv x_3 & = \frac{t^5}{80} |K|^2,
\\
\int_{-t/2}^{t/2}-4 t x_3 (1\pm \d_t n) G_t:\second \dv x_3 & = 2 t^2 \d_t N_{11}:\second 
+ t^2 \d_t^3 n P:\second,
\\
\int_{-t/2}^{t/2} 2 t x_3^2 G_t:K \dv x_3 &= - \frac{t^3\d_t^2}{6} P:K,
\\
\int_{-t/2}^{t/2} 4 x_3^3 (1\pm \d_t n) \second:K \dv x_3 &= \frac{t^4 \d_t}{8}
n \second:K,
\\
\int_{-t/2}^{t/2} 8\d_t^2 |\bm|^2 \dv x_3 &= 8 t\d_t^2 |\bm|^2.
\end{aligned}
\end{equation*}
It remains to examine the forcing term $\bef_t$, which we assume
to be of the form
\[
\bef_t(\bx',x_3) = t^2 \hat\bef(\bx',x_3)
\]
to give a nontrivial limit. In fact, if we let
\[
\bef (\bx') := \frac{1}{t} \int_{-t/2}^{t/2} \bef(\bx',x_3) \dv x_3,
\quad
\bg (\bx') :=  \frac{1}{t} \int_{-t/2}^{t/2} x_3 \bef(\bx',x_3) \dv x_3
\]
then the contribution to the energy due to the body force becomes
\[
\frac{1}{t^3}\int_{\omega_t} \bef_t\cdot\bu \dv\bx = \int_\omega
\Big(\bef(\bx')\cdot\by(\bx')  + \bg(\bx')\cdot\bb(\bx') \Big) \dv\bx'.
\]
Inserting these expressions back into $I_t[\bu]$, setting
\[
\lambda := \lim_{t\to 0} \frac{\d_t}{t} \in \mathbb R
\]
and keeping only terms of order one in $t$, we readily obtain 
\[
\lim_{t\to 0} I_t[\bu] \approx
\frac{1}{12} \int_\o \Big( |\second|^2 + 6 \lambda N_{11}:\second \Big) \dv \bx'
+ \lambda^2 \int_\o \Big( |N_{11}|^2 + 2 |\bm|^2 \Big) \dv \bx'
- \int_\o \bef\cdot\by \dv \bx'.
\]
If we further denote
\begin{equation}\label{spontaneous-curvature}
Z := 3\lambda N_{11}
\end{equation}
and ignore the second integral, which is constant and so independent of
the surface $\gamma$, we see that the dimensionally reduced model is governed
by the energy
\[
E[\by] =  \frac{1}{12} \int_\o \big| \second + Z \big|^2\dv \bx'
- \int_\o \bef\cdot\by \dv \bx',
\]
where the parametrization $\by:\o\to\R^3$ of the surface $\gamma$ is an
isometry, namely it satisfies \eqref{isometry}.
\cpam{We remark that the derivation of the dimensionally reduced model
can be carried out rigorously in the sense of $\Gamma$-convergence 
for a large class of isotropic energy densities~\cite{Schm:07b}. The
only required assumptions are the cubic energy scaling
\eqref{cubic-scaling} and the proportionality $\d_t \sim t$ of \eqref{t-delta}.} 
The quantity $-Z$ acts as a spontaneous curvature for the bending
energy $E[\by]$ and \cpam{specifies} properties of the bilayer material. If the
material is homogeneous and isotropic, then $Z=\alpha I_2$ with
$\alpha\in\R$; we refer to~\cite{Schm:07a} for a discussion of the 
qualitative properties of minimizers. 
On the other hand, the material could possess inhomogeneities and
anisotropies which are $\bx'$-dependent and are encoded in $N_{11}(\bx')$;
we discuss some options together with numerical experiments in \S
\ref{S:experiments}. We
observe that the components $n$ and $\bm$ of $N$ play no role in the
reduced energy.

We assume that the plate is subject to clamped boundary
  conditions on a portion $\partial_D\o$ of $\partial \o$
$$
\by = \by_D, \qquad \nabla \by = \Phi_D \qquad \text{on} \quad \partial_D \o,
$$
where $\by_D:\omega \rightarrow \mathbb R^3$, 
$\Phi_D:\omega \rightarrow \R^{3\times2}$ are sufficiently smooth,
and $\Phi_D=\nabla\by_D$ is an isometry in $\o$, i.e.
$\Phi_D(\bx')^T\Phi_D(\bx')=I_2$ for $\bx'\in\o$.
The variational formulation of the reduced plate model consists of
finding $\by\in\mathcal{A}$, defined in \eqref{admin-set}, such that 
\begin{equation}\label{e:energy}
E[\by] = \frac{1}{2} \int_\o \big| \second + Z \big|^2\dv \bx' - \int_\o \bef \cdot \by \dv{\bx'}.
\end{equation}
is minimized, where $H$ is the second fundamental form defined in
\eqref{fundamental-forms} and $Z$ the spontaneous curvature of
\eqref{spontaneous-curvature}. The new scaling $\frac12$ is immaterial
and just set for convenience. Existence of solutions of the constrained 
minimization problem is a consequence of the direct method in the calculus
of variations.

\section{Kirchhoff Elements on Quadrilaterals}\label{S:Kirchhoff}
The fourth order nature of \eqref{e:energy} and the pointwise
constraint \eqref{isometry} on gradients in the bilayer
bending problem reveal that a careful choice of finite element spaces 
for spatial discretization is mandatory.
\cpam{To avoid $C^1$-elements, which are natural in $H^2$ but difficult to implement},
we employ a nonconforming method that introduces a discrete gradient
operator and which allows us to impose the constraint \eqref{isometry}
at the vertices
of elements. The components of the discrete deformation $\by_h$ belong
to an $H^1$ conforming finite element space $\W_h$ and its discrete gradients
to another $H^1$ conforming finite element space $\G_h$. The degrees of freedom
of our numerical method are the deformations and the deformation gradients
at the nodes of the \cpam{partition $\cT_h$ of $\omega$ into rectangles}
which are the vertices of elements. 

\begin{definition}\label{def:triang}
For a conforming partition $\cT_h$ of $\o$ into \cpam{shape-regular, closed
rectangles} with vertices $\cN_h$ 
and edges $\cE_h$ we define the midpoints of elements and edges,
\cpam{the diameters of elements, and the maximal meshsize} by 
\[
\bz_T := \frac14 \sum_{\bz\in \cN_h\cap T} \bz, \quad 
\bz_E := \frac12 \sum_{\bz\in \cN_h\cap E} \bz, \quad
\cpam{h_T := \diam(T), \quad
h = \max_{T\in \cT_h} h_T}
\]
for all $T\in \cT_h$ and all $E\in \cE_h$.
For every $E\in \cE_h$ we let $\bn_E,\bt_E\in \R^2$ be unit
vectors such that $\bn_E$ is normal to $E$ and $\bt_E$ is tangent to $E$. We denote by
$\bz_{E}^1,\bz_{E}^2 \in \cN_h \cap E$ the end-points of $E$ so that
$E=\conv\{\bz_{E}^1,\bz_{E}^2\}$.
\end{definition}

The following definition modifies the well known Kirchhoff triangles 
\cite{BaBaHo:80,Brae:07} to 
quadrilaterals and is related to \cite{BaTa:82}. 
We let $\Q_r(T)$ \cpam{and $\mathbb{P}_r(T)$} denote the set of 
polynomials on $T$ of partial degree $r$ on each variable \cpam{and
of total degree $r$, respectively.}

\begin{definition}
Let $\cT_h=\{T\}$ be a partition of $\o \subset \R^2$ into 
\cpam{rectangles} as in Defintion~\ref{def:triang}.

(i) {\bf Discrete spaces}: Define 
\[\begin{split}
\W_h &:= \big\{ w_h \in C(\overline{\o}): 
 \, w_h|_T \in \Q_3(T) ~ \forall T\in \cT_h, \ 
\nabla w_h \mbox{ continuous in } \cN_h, \\ 
& \qquad \qquad \qquad  \nabla w_h(z_E) \cdot \bn_E
= \frac12 \big( \nabla w_h(z_{E}^1)+\nabla w_h(z_{E}^2)\big) \cdot
\bn_E ~\forall E\in \cE_h, \big\}, \\
\G_h &:= \big\{ \psi_h \in C(\overline{\o})^2:  \, \psi_h|_T \in
\Q_2(T)^2 ~\forall T\in \cT_h \big\}.
\end{split}\]
(ii) {\bf Interpolation operator}:
Let \cpam{$\hcI_h^2: H^2(\o)^2\to \G_h$} be defined by
\[\begin{aligned}
\hcI_h^2 \bpsi (\bz)  &= \bpsi(\bz) &&  \mbox{ for all $\bz \in \cN_h$}, \\
\hcI_h^2 \bpsi (\bz_E) &= \bpsi(\bz_E) && \mbox{ for all $E\in \cE_h$}, \\
\hcI_h^2 \bpsi (\bz_T) &= \frac14 \sum_{z\in \cN_h\cap T} \bpsi(\bz) && \mbox{ for all $T\in \cT_h$}.
\end{aligned}\]
\cpam{The operator $\hcI_h^2$ is also well-defined for
  discrete vector fields $\psi \in \nabla \W_h$.} \\
(iii) {\bf Discrete gradients}: Define \cpam{$\nabla_h : H^3(\o)\to \G_h$ by 
\[
\nabla_h w := \hcI_h^2 \big[\nabla w\big].
\]
The operator $\nabla_h$ is also well-defined for discrete functions $w \in \W_h$.}
\end{definition}

\cpam{Since the set of vertices, midpoints of edges, and element midpoint is
unisolvent for the polynomial space $\Q_2(T)$, the interpolation operator $\hcI_h^2$ is
well-defined. However, $\hcI_h^2$
differs from the canonical nodal interpolation operator because of the
condition at the element midpoints. The latter is imposed for practical purposes
but makes $\hcI_h^2$ inexact over $\Q_2(T)$.
Nevertheless, $\hcI_h^2$ is exact over $\Q_1(T)$ so that
the Bramble-Hilbert lemma implies 
\begin{equation}\label{e:nodal_interp_2}
  \|\bpsi  - \hcI_h^2 \bpsi \|_{L^p(T)} + h_T \| \nabla \bpsi
  - \nabla \hcI_h^2 \bpsi \|_{L^p(T)}
\le c_2 h_T^2 \|D^2 \bpsi \|_{L^p(T)}
\end{equation}
for all $\bpsi \in W^2_p(T)^2$ and $2\le p\le\infty$. 
The operator $\hcI_h^2$ is also well-defined} on $\nabla \W_h$ since for every
$w_h \in \W_h$ we have that $\nabla w_h$ 
is continuous at the nodes $\cN_h$ and at the midpoints of edges. 
We will also need the canonical nodal interpolation operator 
$\hcI_h^3: H^3(\o) \to \W_h$, which \cpam{is defined by evaluating 
function values and derivatives at vertices of elements and normal
derivatives at midpoint of edges by averaging.
Since $\hcI_h^3$ is exact for $w\in\mathbb{P}_2(T)$,
the Bramble-Hilbert lemma yields
\begin{equation}\label{e:nodal_interp_3}
\| w  - \hcI_h^3 w \|_{L^p(T)} + h_T \| \nabla w  - \nabla \hcI_h^3 w \|_{L^p(T)}
+ h_T^2 \| D^2 w - D^2 \hcI_h^3  w \|_{L^p(T)}
\le c_2 h_T^3 \|D^3 w \|_{L^p(T)}
\end{equation}
for all $w \in W^3_p(T)$ and $2\le p \le \infty$.
A less obvious but useful stability bound reads
\begin{equation}\label{stab-bound}
  \|D^3\hcI_h^3  w \|_{L^p(T)} \le c \|D^3 w\|_{L^p(T)}
  \qquad\textrm{for all } w\in W^3_p(T) \text{ and } 2 \le p \le \infty.
\end{equation}
To see this, we first write
$\hcI_h^3  \big(w-q\big)=\big(w-q\big) + \big(\hcI_h^3  w - w\big)$
for all $q\in\mathbb{P}_2(T)$.
Therefore, invoking an inverse estimate together with $D^3q=0$, we
use \eqref{e:nodal_interp_3} to obtain
\begin{equation*}
  \begin{aligned}
    \|D^3 \hcI_h^3  w \|_{L^p(T)} &\le c h_T^{-2} \|\nabla\hcI_h^3 (w-q) \|_{L^p(T)}
    \\
    & \le c h_T^{-2} \|\nabla(w-q)\|_{L^p(T)} +
     c h_T^{-2} \|\nabla(\hcI_h^3 w - w) \|_{L^p(T)}
    \le c \|D^3  w \|_{L^p(T)},
  \end{aligned}  
\end{equation*}  
provided that $q$ is appropriately chosen, e.g., as a suitable Lagrange
interpolant of $w$ over $\mathbb{P}_2(T)$. Hereafter, $c>0$ indicates
a generic geometric constant that may change at each occurrence,
depends on mesh shape regularity, but is
independent of the functions and parameters involved.}

\begin{remark}[nodal degrees of freedom]
\rm
The degrees of freedom in $\W_h$ are only the function values 
at the vertices $(w_h(\bz):\bz\in \cN_h)$, 
and the gradients at the vertices $(\nabla w_h(\bz):\bz \in \cN_h)$.
In fact, the remaining four degrees of freedom of the finite element $\Q_3(T)$ 
are the normal components $\nabla w_h(\bz_E)$ of the gradients
at the midpoints $\bz_E$  of edges $E$ which are fixed as the averages
of directional derivatives $\nabla w_h(\bz_E^i)\cdot\bn_E$ 
at the endpoints $\bz_E^i$ of the edges. 
The values $\nabla w_h(\bz_E)\cdot\bt_E$ can be written in terms
  of $w_h(\bz_E^i)$ and $\nabla w_h(\bz_E^i)\cdot\bt_E$ for
  $i=1,2$. The matrix realizing the operator $\nabla_h : \W_h\to\G_h$
    elementwise is required for the implementation of Kirchhoff elements.
\end{remark}

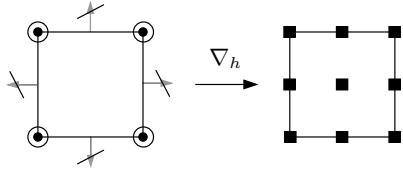
\begin{figure}[htb]
\begin{center}
\begin{picture}(0,0)%
\includegraphics{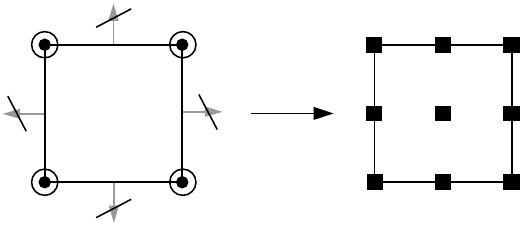}%
\end{picture}%
\setlength{\unitlength}{3947sp}%
\begingroup\makeatletter\ifx\SetFigFont\undefined%
\gdef\SetFigFont#1#2#3#4#5{%
  \reset@font\fontsize{#1}{#2pt}%
  \fontfamily{#3}\fontseries{#4}\fontshape{#5}%
  \selectfont}%
\fi\endgroup%
\begin{picture}(2503,1078)(509,-602)
\put(1904, 59){\makebox(0,0)[b]{\smash{{\SetFigFont{9}{10.8}{\familydefault}{\mddefault}{\updefault}{\color[rgb]{0,0,0}$\nabla_h$}%
}}}}
\end{picture}%
\end{center}
\caption{\label{fig:discre_grad}
\small
Schematic description of the discrete gradient operator $\nabla_h$.
Filled dots represent values of functions, circles of gradients,
arrows of normal components, and boxes of vector fields. The
normal derivatives in the cubic space on the left are eliminated
via linearity.}
\end{figure}

\begin{remark}[subspaces of $H^1(\omega)$]
\rm
Enforcing degrees of freedom of $\W_h$ at vertices $\bz\in\cN_h$ for
both function values and gradients implies global continuity; thus
$\W_h\subset H^1(\omega)$. Likewise, the degrees of freedom of $\G_h$
at vertices and midpoints of edges guarantee global continuity; hence
$\G_h\subset H^1(\omega)^3$.
\end{remark}

We collect important properties of the discrete gradient operator in the
following proposition. 

\begin{proposition}[properties of $\nabla_h$]
\cpam{Let $2\le p\le \infty$.}
There are constants $c_i$, $i=1,...,4$, independent of $h$ such that
the following properties of the discrete gradient $\nabla_h$ are valid:

\smallskip
\cpam{
(i) For all $w_h\in \W_h$ we have
\begin{equation}\label{norm-equiv}
c_1^{-1} \|\nabla w_h \|_{L^p(\omega)} \le \|\nabla_h w_h \|_{L^p(\omega)}  \le c_1 \|\nabla w_h\|_{L^p(\omega)};
\end{equation}
(ii) For all $w_h \in \W_h$ and $T\in \cT_h$ we have 
\begin{equation}\label{hessian-equiv}
c_2^{-1} \|D^2 w_h\|_{L^p(T)} \le \|\nabla\nabla_h w_h\|_{L^p(T)}
\le c_2  \|D^2 w_h\|_{L^p(T)};
\end{equation}
(iii) For all {$w \in W^3_p(T)$} and $T\in \cT_h$ we have 
\begin{equation}\label{approx}
\|\nabla w - \nabla_h w \|_{L^p(T)} + h_T \| D^2 w - \nabla \nabla_h w \|_{L^p(T)}
\le c_3 h_T^2 \|D^3 w\|_{L^p(T)};
\end{equation}
(iv) For all $w_h \in \W_h$ and $T\in \cT_h$ we have
\begin{equation}\label{discrete-approx}
\|\nabla w_h - \nabla_h w_h \|_{L^p(T)} \le c_4 h_T \|\nabla\nabla_h w_h \|_{L^p(T)}.
\end{equation}
}
\end{proposition}

\begin{proof}
(i) Given $w_h\in \W_h$ the function $\bpsi_h=\nabla_h w_h \in \G_h$ is well-defined
and the operator $\nabla_h :\W_h \to \G_h$ is linear, whence $\nabla w_h=0$
implies $\nabla_h w_h=0$. Conversely, if $\nabla_h w_h =0$ then we have that 
$\nabla w_h(\bz)=0$ for all $\bz\in \cN_h$ and $\nabla w_h (\bz_E) =0$ for all 
$E\in \cE_h$.
Since the tangential derivatives of $w_h$ vanish at the endpoints
  and midpoints of $E\in\cE_h$, and $w_h$ is cubic on $E$, we deduce
  that $w_h$ is constant on $E$. The fact that functions in $\W_h$ are
  globally continuous implies that $w_h$ is constant over the skeleton
  $\cE_h$ of $\cT_h$. Let $T\in\cT_h$ and note that there are four
  remaining  degrees of freedom in $\Q_3(T)$. Since $\nabla
  w_h(\bz_E)\cdot\bn_E=0$ for all $E\in\cE_h\cap T$, we see that $w_h$
  is constant in $T$, whence $\nabla w_h = 0$.
The equivalence of the identities $\nabla w_h=0$ and
 $\nabla_h w_h=0$ implies the asserted norm equivalence because
   $\Q_3(T)$ is finite dimensional. \\
(ii) We proceed as in (i). If $D^2 w_h=0$, then $\nabla w_h$ is
constant and so is $\nabla_h w_h$ according to its definition; thus 
$\nabla\nabla_h w_h=0$. Conversely, if $\nabla\nabla_h w_h=0$, then 
$\nabla_h w_h$ is constant in $T$ and thus $\nabla w_h$ is the same
constant at the vertices and midpoints of edges of $T$. This matches
the 16 degrees of freedom of $\Q_3(T)$, whence $\nabla w_h$ is constant
in $T$ and $D^2 w_h = 0$.
\\
(iii) 
{Estimate \eqref{approx} follows from the interpolation estimate \eqref{e:nodal_interp_2} with $\bpsi=\nabla w$ upon 
noting that $\hcI_h^2 [\nabla w]
= \nabla_h w$. } \\
(iv) The estimate \eqref{discrete-approx} is a consequence of
\eqref{approx}, an inverse inequality, and \eqref{hessian-equiv}. \\
\cpam{The independence of all constants of the element-size $h_T$
follows from scaling arguments.}
\end{proof}

\begin{remark}[bases of $\W_h$ and $\G_h$]\label{r:basis}
\rm 
We anticipate that our discrete algorithms 
(Algorithms \ref{alg:grad_flow_discr} and \ref{alg:fixed_point} below) 
do not require the choice of a particular basis for $\mathbb W_h$. 
Instead, we apply vertex based quadratures requiring only the values 
of the approximate deformation and its gradient at the vertices.
In contrast, a basis for $\G_h$ is required but standard. \cpam{In
our implementation we use the (nodal) Lagrange basis.} 
\end{remark}

\section{Discrete Energies and $\Gamma$-Convergence of the Discretization}\label{S:convergence} 

We employ the
Kirchhoff elements on quadrilaterals $\W_h^3\subset H^1(\o)^3$ and the discrete gradient
operator $\nabla_h:\W_h^3 \to \G_h^3$, whose components are denoted $\p_j^h$, $j=1,2$, to
approximate the energy $\tE$ given by \eqref{e:energy-used}.
For practical purposes, we also impose a relaxed isometry constraint
at the vertices of elements, \cpam{but we introduce a}
parameter $\delta \geq 0$ to control its violation.
We will show in Section~\ref{S:gradient-flow} that in the context of
\cpam{an $H^2$-gradient flow}, $\delta$ is 
proportional to the gradient flow pseudo-timestep and can therefore be
made arbitrary small. \cpam{We next give a discrete
version of \eqref{admin-set} and \eqref{admin_tangent}.}

\begin{definition}
For $\delta \geq 0$, $\by_{D,h} \in \W_h^3$ and $\Phi_{D,h} \in \G_h^3|_{\p_D \O}$ 
let the \emph{discrete admissible \cpam{set}} be
\[\begin{split}
\cA_h^\delta := \big\{  \by_h \in \W_h^3  : \, & \by_h|_{\partial_D \o} = \by_{D,h}|_{\p_D \o}, \ 
{\nabla_h \by_h|_{\partial_D \o}}  =\Phi_{D,h}|_{\p_D \o}, \\
& [\nabla \by_h(\bz)]^\transp \nabla \by_h(\bz) \ge I_2 \ \forall \bz \in \mathcal N_h, \
 \|[\nabla \by_h]^\transp \nabla \by_h - I_2  \|_{L^1_h(\o)} \le \delta \big\}.
\end{split}\]
The \emph{\cpam{(pseudo)} tangent space of $\cA_h^\delta$ at $\by_h\in
  \cA_h^\delta$} is defined by 
\[\begin{split}
\cF_h[\by_h]:=\big\{  \bw_h \in \W_h^3 : \,
 \bw_h|_{\partial_D \omega} = 0,   \ \nabla & \bw_h|_{\partial_D \o} =0, \\
& [\nabla \bw_h(\bz)]^T\nabla \by_h(\bz)+ [\nabla \by_h(\bz)]^T\nabla \bw_h(\bz) =0 \ 
 \forall \bz \in \mathcal N_h \big\}.
\end{split}\]
\end{definition}

\cpam{Notice that $\cF_h[\by_h]$ would be the tangent space to
$\cA_h^\delta$ at $\by_h$ if $[\nabla\by_h(\bz)]^T \nabla\by_h(\bz)$
were constant for vector fields in $\cA_h^\delta$ 
at every node $\bz\in\mathcal N_h$, which explains the terminology.
Notice also} the use of the discrete norms $\|\phi\|_{L^p_h(\o)}$, which for
$1\le p<\infty$ are defined by 
\[
\|\phi\|_{L^p_h(\o)}^p  :=  
\sum_{T\in \cT_h}  \frac{|T|}{4} \sum_{\bz \in \cN_h \cap T} \big|\phi|_T(\bz)\big|^p,
\]
and satisfy the equivalence relation $\|v_h\|_{L^p(\o)} \sim  \|v_h\|_{L^p_h(\o)}$
for piecewise bilinear functions $v_h \in C(\overline{\o})$. 
We also define the discrete inner product $(\cdot,\cdot)_h$ for 
piecewise continuous functions $\phi,\psi\in \Pi_{T\in\cT_h} C^0(T)$
\[
(\phi,\psi)_h := \sum_{T\in \mathcal T_h} \frac{|T|}{4} \sum_{\bz \in \cN_h \cap T} \phi|_T(\bz) \psi|_T(\bz)
\]
and note that $\|\phi \|_{L^p_h(\o)}^p= (|\phi|^p,1)_h$.

The finite element discretization $\tE_h$ of the energy functional $\tE$ 
in \eqref{e:energy-used} is given by 
\begin{equation}\label{e:discrete_energy}
\tE_h[\by_h] :=  \frac12 \int_\o |\nabla \nabla_h \by_h|^2  
+ \sum_{i,j=1}^2 \Big(\p_i\cI^1_h[\p_j^h \by_h] \cdot 
\Big[\frac{\p_1^h \by_h}{|\p_1^h \by_h|} \times \frac{\p_2^h \by_h}{|\p_2^h \by_h|}\Big], Z_{ij} \Big)_h 
+ \frac12 \cpam{(Z,Z)_h} - (\bef, \by_h)_h,
\end{equation}
for {$\by_h\in\mathcal{A}_h^\delta$} and $\tE_h[\by_h]=\infty$ otherwise,
where $\cI^1_h$ is the \cpam{canonical} Lagrange
interpolation operator into the continuous piecewise $\Q_1$ elements,
and both $Z$ and $\bef$ are piecewise continuous in $\bar\omega$.
\cpam{The latter enables the use of quadrature for the last three terms,
whereas the first term can be integrated exactly because $\nabla_h\by_h$ is
piecewise $\Q_2$. The energy \eqref{e:discrete_energy} is thus practical.}

\begin{remark}[discrete isometry relation]
\rm 
The nodal isometry relation 
$[\nabla\by_h(\bz)]^T\nabla\by_h(\bz) \ge I_2$ for $\by_h\in \cA_h$ implies that
$|\p_j^h \by_h(\bz)|\ge 1$ for $j=1,2$ and all $\bz\in\cN_h$. Hence, the normalization
$\p_j^h \by_h(\bz)/|\p_j^h \by_h(\bz)|$ in the discrete energy
functional $\tE_h[\by_h]$ is well-defined. We will see that it allows  
for suitable energy bounds and gives rise to a coercivity property.
\end{remark}

We start by showing that the family $\{\tE_h \}_{h\geq 0}$ is (equi-)coercive.

\begin{proposition}[coercivity]\label{p:coercivity}
Let the Dirichlet boundary data satisfy $\by_D \in H^3(\omega)^3$ and
$\Phi_D\in H^2(\omega)^{3\times 2}$, and let $\by_{D,h}:= \cI_h^3 \by_D$, $\Phi_{D,h}:= \cI_h^2 \Phi_D$.
{Let the data satisfy $Z \in \Pi_{T\in\mathcal{T}_h} C^0(T)^{2\times2}$
  and $\bef \in \Pi_{T\in\mathcal{T}_h} C^0(T)^3$.}
Let $\{\by_h\}_{h>0}$ be a sequence of displacements in $H^1(\o)^3$ such that 
for a constant $C$ independent of $h$ there holds
$$
\tE_h [\by_h] \leq C.
$$
Then $\by_h \in \cA_h^\delta$ and there exists a constant $\widetilde C$
depending on $\| Z \|_{L^\infty(\o)}$,
{$\|\bef\|_{L^\infty(\omega)}$,} $\|\by_D\|_{H^3(\omega)}$, and $\|\Phi_D \|_{H^2(\omega)}$, but independent of $h$, such that
\begin{equation}\label{apriori-H2}
\| \nabla \nabla_h \by_h \|_{L^2(\o)} \leq \widetilde C.
\end{equation}
\end{proposition}
\begin{proof}
We first argue that $\by_h\in \cA_h^\delta$ since otherwise $\tE_h [\by_h] =+\infty$. 
As a consequence we have $|\p_j^h \by_h(\bz)|\ge 1$ for $j=1,2$ 
and all $\bz\in \cN_h$, whence there exists a constant $c$ independent of $h$ such that
\begin{equation}\label{coercivity}
\tE_h[\by_h] \ge \frac{1}{2} \|\nabla \nabla_h \by_h\|^2_{L^2(\o)}
- c  \big( \|\cpam{\nabla} \cI^1_h[ \nabla_h \by_h]\|_{L^2(\o)} \|Z\|_{L^\infty(\o)} + \| \by_h \|_{L^2(\o)} \|\bef\|_{L^\infty(\o)} \big).
\end{equation}
Since $\by_h = \by_{D,h}$ and $\nabla_h \by_h = \Phi_{D,h}$ on $\partial_D \omega$, we can apply the Poincar\'e inequality twice and bound $\| \by_h \|_{L^2(\omega)}$ in terms of $\| \nabla \nabla_h \by_h \|_{L^2(\omega)}$, $\| \by_{D,h} \|_{H^1(\omega)}$, and $\| \Phi_{D,h}\|_{H^1(\omega)}$.
In view of  \eqref{e:nodal_interp_3} and \eqref{e:nodal_interp_2}, the
latter two {quantities} are bounded by a constant times $\| \by_D \|_{H^3(\omega)}$ and $\| \Phi_D \|_{H^2(\omega)}$, respectively.
We observe that {for all $T\in\mathcal{T}_h$
\begin{equation}\label{bound-Il2}
\|\nabla\cI^1_h[\nabla_h \by_h]\|_{L^2(T)} \le c h_T 
\|\nabla\cI^1_h[\nabla_h \by_h]\|_{L^\infty(T)} \le c h_T
\|\nabla \nabla_h \by_h\|_{L^\infty(T)} \le c \|\nabla \nabla_h \by_h\|_{L^2(T)},
\end{equation}
where the last step is an inverse inequality for $\nabla_h\by_h\in\G_h$.
This implies the asserted bound.}
\end{proof}
\begin{remark}[coercivity and gradient flows]
\rm 
In the (energy decreasing) gradient flow  setting adopted in Section~\ref{S:gradient-flow}, the assumptions of Proposition~\ref{p:coercivity} are automatically satisfied provided the initial state has finite energy; see Proposition~\ref{prop:grad_flow}. 
\end{remark}

We now show $\Gamma$-convergence of $\tE_h$ to \cpam{$\tE$ in $H^1(\o)^3$}
and deduce the accumulation of almost \cpam{global} minimizers of $\tE_h$ at
global minimizers of the continuous problem. 
For this, we assume that the discrete boundary
conditions are obtained by interpolation of the continuous ones
with strong convergence in $L^2(\p_D\o)$. We also assume for
simplicity that $Z$ and $\bef$ are piecewise constant. 

\begin{theorem}[$\Gamma$-convergence]\label{T:gamma-convergence}
Let the Dirichlet boundary data satisfy $\by_D \in H^3(\omega)^3$ and $\Phi_D\in H^2(\omega)^{3\times 2}$, and let $\by_{D,h}:= \cI_h^3 \by_D$, $\Phi_{D,h}:=\cI_h^2 \Phi_D$.
If $Z$ and $\bef$ are piecewise constant over the partition
$\cT_h$, then the following two properties hold:

\smallskip
{\em (i) Attainment.} For all $\by \in \cpam{\cA}$, there exists a
sequence {$\{ \by_h \}_h$}
with $\by_h \in \cA_h^0 \subset \cA_h^\delta$ for all $h>0$ such that $\by_h \to \by$ in $H^1(\o)^3$ and  
\[
\limsup_{(h,\delta) \to 0} \tE_h[\by_h] \leq \cpam{\tE}[\by].
\]

\smallskip
{\em (ii) Lower bound property.} Assume that $\delta \to 0$ as $h \to 0$. 
For all $\by \in H^1(\o)^3$ and all sequences $\{ \by_h \} \subset H^1(\omega)^3$ 
such that $\by_h \to \by$ in $H^1(\o)^3$, we have
\[
\cpam{\tE}[\by] \le \liminf_{h\to 0} \tE_h[\by_h].
\]
\end{theorem}

\begin{proof}
  We prove properties (i) and (ii) separately.

\smallskip\noindent
(i) \cpam{Since $\by\in\cA\subset H^2(\o)^3$},
for every $\epsilon>0$ the density of smooth isometries among isometries in $H^2(\o)^3$, cf.~\cite{Horn:11}, implies the existence of an isometry $\by_\epsilon \in H^3(\omega)^3$ such that
\begin{equation}\label{e:epsilon_approx}
\| \by - \by_\epsilon\|_{H^2(\o)} \leq \epsilon.
\end{equation}
This, in conjunction with the isometry property of
both $\by$ and $\by_\epsilon$, yields
\[
\big| \cpam{\tE} [\by] - \cpam{\tE} [\by_\epsilon] \big| \le C\epsilon.
\]
Therefore, we assume from now on that $\by\in H^3(\omega)^3$ and do
not write the subscript $\epsilon$ for simplicity. 

For $h>0$ let $\by_h = \hcI_h^3\by \in \W_h^3$ be the nodal interpolant of $\by$, i.e.,
we have $\by_h(\bz)=\by(\bz)$ and $\nabla \by_h (\bz)=\nabla \by(\bz)$ for
all $\bz \in \cN_h$. The latter yields $[\nabla \by_h]^\transp \nabla \by_h = I_2$
at the nodes in $\cN_h$, whence $\by_h \in \cA_h^0$. 
Convergence of $\by_h$ to $\by$ in $H^1(\omega)$ directly follows from the
interpolation estimate \eqref{e:nodal_interp_3}
$$
\| \by - \by_h \|_{H^1(\o)} \leq c_2  h^2 \| \by \|_{H^3(\o)}.
$$

It thus remains to prove the convergence
of the discrete energies $\tE_h[\by_h]$ to $\tE[\by]$. 
To derive the convergence of the first  term in \eqref{e:discrete_energy}, we write
\begin{equation}\label{grad-rel}
\nabla\by - \nabla_h\by_h =  \nabla \big(\by-\hcI_h^3\by\big)
+ \big(\nabla \hcI_h^3\by - \hcI_h^2[\nabla \hcI_h^3\by]\big)
\end{equation}
\cpam{and use \eqref{approx} in conjunction with \eqref{stab-bound} to get
\[
\|\nabla\big( \nabla\cI_h^3\by-\cI_h^2[\nabla\cI_h^3\by] \big)\|_{L^2(T)}
\le ch_T \|D^3\cI_h^3\by\|_{L^2(T)}
\le c h_T \|D^3\by\|_{L^2(T)}.
\]
Combining this with \eqref{e:nodal_interp_3} yields
}
\begin{equation}\label{strong_conv}
\|\nabla \nabla_h \by_h - D^2 \by \|_{L^2(\o)}
 \le c h \|D^3 \by \|_{L^2(\o)}.
\end{equation}
For the second term in $\tE_h[\by_h]$, we first note
that {$|\p_j^h \by_h(\bz)|=1$ for $j=1,2$ and all $\bz\in \cN_h$,} and by nodal
interpolation estimates
\begin{equation}\label{quad}
\begin{aligned}
\Big|\big( \p_i \cI^1_h[\p_j^h \by_h] \cdot \big[\p_1^h \by_h 
& \times \p_2^h \by_h \big], Z_{ij} \big)_h 
- \big(\p_i \cI^1_h[\p_j^h \by_h] \cdot \big[\p_1^h \by_h 
\times \p_2^h \by_h \big], Z_{ij} \big) \Big| \\
& \le c \sum_{T\in \cT_h} h_T^2  \big\|D \big( \p_i 
\cI^1_h[\p_j^h \by_h] ) \|_{L^2(T)}
\|D \big[\p_1^h \by_h \times \p_2^h \by_h \big] \|_{L^2(T)}\\
& + c \sum_{T\in \cT_h} h_T^2  \big\| \p_i 
\cI^1_h[\p_j^h \by_h] \|_{L^2(T)}
\|D^2 \big[\p_1^h \by_h \times \p_2^h \by_h \big] \|_{L^2(T)},
\end{aligned}
\end{equation}
because $D^2 \p_i \cI^1_h[\p_j^h \by_h]=0$ for every $T\in \cT_h$
and $Z$ is piecewise constant over $\cT_h$;
\cpam{recall that} $c$ denotes a generic constant independent of $h$.
Therefore, employing inverse estimates for both terms on the right-hand
side of the preceding estimate, {and recalling \eqref{bound-Il2},}
we deduce
\begin{align*}
\Big|\big( \p_i \cI^1_h[\p_j^h \by_h] &\cdot \big[\p_1^h \by_h 
\times \p_2^h \by_h \big], Z_{ij} \big)_h 
- \big(\p_i \cI^1_h[\p_j^h \by_h] \cdot \big[\p_1^h \by_h 
\times  \p_2^h  \by_h \big], Z_{ij} \big) \Big| \\
& \le c \sum_{T\in \cT_h} h_T  \big\| \nabla\nabla_h \by_h \|_{L^2(T)}^2
\|\nabla_h \by_h\|_{L^\infty(T)}
\le c h \big\| \nabla\nabla_h \by_h \|_{L^2(\omega)}^2
\|\nabla_h \by_h\|_{L^\infty(\omega)}.
\end{align*}
We further observe that $\|\nabla_h\by_h\|_{L^\infty(\o)}$ is bounded
  uniformly because $\by_h = \hcI_h^3\by$ \cpam{with $\by \in
  H^3(\o)^3\subset C^1(\overline\o)$ being an isometry, and \eqref{norm-equiv} with
  $p=\infty$.} 
 This, together with 
\eqref{strong_conv}, implies that the quadrature term above
is bounded by {$ch \cpam{\|\by\|_{H^3(\omega)}^2}$.
It thus remains to examine
\begin{equation*}
\big( \p_i \cI^1_h[\p_j^h \by_h] \cdot \big[\p_1^h \by_h 
  \times \p_2^h \by_h \big], Z_{ij} \big) -
\big( \p_i \p_j \by \cdot \big[\p_1 \by 
  \times \p_2 \by \big], Z_{ij} \big).
\end{equation*}
Invoking again~\eqref{grad-rel}, we infer that
$\|\nabla\by -\nabla_h\by_h \|_{L^2(\o)} \le ch^2 \| D^3\by \|_{L^2(\o)}$}
along with
\begin{align*}
\|\big[\p_1^h \by_h \times \p_2^h \by_h \big] 
&- \big[\p_1 \by\times \p_2 \by \big]\|_{L^2(\omega)} \\
& \le \|\big[\p_1^h \by_h - \p_1 \by \big] \times \p_2^h \by_h\|_{L^2(\omega)}
+ \|\big[\p_1 \by \times \big[\p_2^h \by_h - \p_2 \by \big]\|_{L^2(\omega)}
{\le ch^2 \| D^3\by \|_{L^2(\o)}}
\end{align*}
because $\|\nabla_h\by_h\|_{L^\infty(\omega)}, \|\nabla\by\|_{L^\infty(\omega)} \le C$.
\cpam{In addition, we see that
\begin{equation*}
  \nabla \cI_h^1[\nabla_h\by_h]-\nabla^2\by
  = \nabla\cI_h^1[\nabla_h\by_h - \nabla\by]
  + \nabla\big(\cI_h^1[\nabla\by] - \nabla\by  \big),
\end{equation*}
along with
\begin{equation*}
  \| \nabla\big(\cI_h^1[\nabla\by] - \nabla\by  \big) \|_{L^2(T)}
  \le c h_T \|D^3\by\|_{L^2(T)}.
\end{equation*}  
Using an inverse estimate and stability of $\cI_h^1$ in $L^\infty(T)$,
we get
\begin{equation*}
  \begin{aligned}
    \|\nabla\cI_h^1[\nabla_h\by_h - \nabla\by]\|_{L^2(T)}
    & \le c h_T^{-1} \|\cI_h^1[\nabla_h\by_h - \nabla\by]\|_{L^2(T)} \\
    & \le c \|\cI_h^1[\nabla_h\by_h - \nabla\by]\|_{L^\infty(T)}
    \le c \|\nabla_h\by_h - \nabla\by\|_{L^\infty(T)}.
  \end{aligned}
\end{equation*}  
Moreover, we further write
\begin{align*}
  \|\nabla_h\by_h - \nabla\by\|_{L^\infty(T)} &=
  \|\cI_h^2[\nabla\cI_h^3\by] - \nabla\by\|_{L^\infty(T)} \\
  &\le \|\cI_h^2[\nabla\cI_h^3\by] - \nabla\cI_h^3\by\|_{L^\infty(T)}
  + \|\nabla(\cI_h^3\by-\by)\|_{L^\infty(T)},
\end{align*}
and obtain, according to \eqref{e:nodal_interp_2} and \eqref{stab-bound} with
$p=\infty$ and an inverse estimate,
\begin{equation*}
  \|\cI_h^2[\nabla\cI_h^3\by] - \nabla\cI_h^3\by\|_{L^\infty(T)}
  \le ch_T \|D^3 \cI_h^3\by\|_{L^2(T)} \le c h_T \|D^3\by\|_{L^2(T)}.
\end{equation*}  
Since $\|\nabla(\cI_h^3\by-\by)\|_{L^\infty(T)} \le ch_T \|D^3\by\|_{L^2(T)}$,
we deduce $\|\nabla \cI_h^1[\nabla_h\by_h]-\nabla^2\by\|_{L^2(T)} \le
c h_T \|D^3\by\|_{L^2(T)}.$
This, together with the preceding bound for
$\|\big[\p_1^h \by_h \times \p_2^h \by_h \big] 
- \big[\p_1 \by\times \p_2 \by \big]\|_{L^2(\omega)}$,
implies
}
\begin{equation}\label{e:termIII}
\left| \big( \p_i \cI^1_h[\p_j^h \by_h] \cdot \big[\p_1^h \by_h \times \p_2^h \by_h \big], Z_{ij} \big) - \big( {\p_i \p_j \by} \cdot \big[\p_1 \by 
\times \p_2 \by \big], Z_{ij} \big) \right| \leq c h\|D^3 \by \|_{L^2(\o)}.
\end{equation}
Collecting the preceding estimates, we obtain
$\big| \tE_h[\by_h] - \cpam{\tE}[\by]\big| \leq ch\|D^3 \by \|_{L^2(\o)}$
where $\by$ is an abbreviation for $\by_\epsilon$. Selecting $h=h(\epsilon)$ to be
sufficiently small so that
$h\|D^3 \by_\epsilon \|_{L^2(\o)} \leq \epsilon$ yields
$$
\big| \tE_h[\by_h] - \cpam{\tE}[\by]\big| \leq c\epsilon.
$$
and the convergence of $\tE_h[\by_h]$ to $\cpam{\tE}[\by]$ when $h\to 0$ follows.

\smallskip
(ii) We may assume that $\by_h \in \cA_h^\delta$ and $\tE_h[\by_h] \leq C$ uniformly in $h$ (perhaps for a subsequence not relabeled) for otherwise $\liminf_{h\to 0} \tE_h[\by_h] =+\infty$ and there is nothing to prove.
Hence Proposition~\ref{p:coercivity} implies that the
sequence $\{\nabla_h \by_h\}_{h>0}$ is uniformly bounded in $H^1(\o)^{3\times 2}$. 
This guarantees the existence of $\Phi\in H^1(\o)^{3\times 2}$ 
such that after extraction of a subsequence (not relabeled) we have
{$\Phi_h=\nabla_h \by_h \rightharpoonup \Phi$ in \cpam{$H^1(\o)^{3\times 2}$} and
$\Phi_h \to \Phi$ in $\cpam{L^2(\o)^{3\times2}}$ as $h \to 0$.
The discrete \cpam{approximation} estimate \eqref{discrete-approx} yields}  
$$
\| \nabla \by - \Phi_h \|_{L^2(\o)} \leq \| \nabla \by - \nabla \by_h \|_{L^2(\o)} + \| \nabla \by_h - \nabla_h \by_h \|_{L^2(\o)} 
\leq  \| \nabla \by - \nabla \by_h \|_{L^2(\o)} + ch \| \nabla \nabla_h \by_h\|_{L^2(\o)},
$$
whence taking the limit when $h\to 0$ we deduce $\Phi=\nabla \by$ and $\by \in H^2(\o)^3$
because $\nabla\by_h\to\nabla\by$ in \cpam{$L^2(\omega)^{3\times 2}$
by assumption.}
Owing to the assumptions on the boundary data we have that
$\by|_{\p_D\o} = \by_\DD$ and $\nabla \by|_{\p_D\o} = \Phi_\DD$. To show that 
$\by$ is an isometry we {utilize discrete interpolation estimates
and $\by_h\in\cA_h^\delta$}
\begin{align*}
\|\Phi_h^\transp \Phi_h - \Id_2 \|_{L^1(\o)}
&\le  \|\Phi_h^\transp \Phi_h - \cI_h^1[\Phi_h^\transp \Phi_h] \|_{L^1(\o)}
+  \| \cI_h^1[\Phi_h^\transp \Phi_h]- \Id_2  \|_{L^1(\o)}
\\
&\le c h \| \nabla [\Phi_h^\transp \Phi_h] \|_{L^1(\o)} + c_0 \delta.
\end{align*}
The right-hand side converges to zero as $(h,\delta)\to 0$ because
  of the uniform bound {\eqref{apriori-H2}} of $\Phi_h$ in $H^1(\o)^{3\times 2}$. 
Hence, $\Phi_h^\transp \Phi_h \to \Id_2$ pointwise almost
everywhere in $\o$ for an appropriate subsequence and, 
since $\Phi_h^\transp \Phi_h \to \Phi^\transp \Phi$ pointwise almost
everywhere in $\o$, we deduce that $\by$ is an isometry a.e. in $\o$, i.e., $\by\in \cA$. 
Since the $H^1$-seminorm
is weakly lower semicontinuous we get 
{$\int_\omega |D^2 \by|^2 = \int_\o |\nabla\Phi|^2\le \liminf_{h\to0} \int_\omega
|\nabla\Phi_h|^2$.} It remains to prove that
the following three terms tend to $0$:
\begin{gather*}
I_h = \Big(\p_i\cI^1_h[\p_j^h \by_h] \cdot 
\Big[\frac{\p_1^h \by_h}{|\p_1^h \by_h|} 
\times \frac{\p_2^h \by_h}{|\p_2^h \by_h|}\Big], Z_{ij} \Big)_h
-
\Big(\p_i\cI^1_h[\p_j^h \by_h] \cdot 
\big[\p_1^h \by_h \times \p_2^h \by_h \big], Z_{ij} \Big)_h , 
\\
II_h = \Big(\p_i\cI^1_h[\p_j^h \by_h] \cdot 
\big[\p_1^h \by_h \times \p_2^h \by_h \big], Z_{ij} \Big)_h
-
\Big(\p_i\cI^1_h[\p_j^h \by_h] \cdot 
\big[\p_1^h \by_h \times \p_2^h \by_h \big], Z_{ij} \Big) ,
\\
III_h = \Big(\p_i\cI^1_h[\p_j^h \by_h] \cdot 
\big[\p_1^h \by_h \times \p_2^h \by_h \big], Z_{ij} \Big)
-
\Big(\p_i\p_j \by \cdot 
\big[\p_1 \by \times \p_2 \by \big], Z_{ij} \Big) ,
\end{gather*}
for all $1\le i,j \le 2$. We first note that
\[
|I_h| \le \big\| \p_i \cI^1_h[\p_j^h \by_h] \big\|_{L^2_h(\o)} \|Z_{ij}\|_{L^\infty(\omega)}
\Big\|\frac{\p_1^h \by_h}{|\p_1^h \by_h|} 
\times \frac{\p_2^h \by_h}{|\p_2^h \by_h|}- \p_1^h \by_h \times \p_2^h
\by_h\Big\|_{L^2_h(\o)}.
\]
The first factor on the right-hand side is bounded as $h\to 0$
according to {\eqref{bound-Il2} and \eqref{apriori-H2},} and the second one by assumption.
Since $|\p_i^h\by_h|\ge 1$, we estimate the last factor as follows:
\[\begin{split}
\Big\|\frac{\p_1^h \by_h}{|\p_1^h \by_h|} 
&\times \frac{\p_2^h \by_h}{|\p_2^h \by_h|}- \p_1^h \by_h \times \p_2^h \by_h\Big\|_{L^2_h(\o)} \\
&\le \Big\|\Big(\frac{\p_1^h \by_h}{|\p_1^h \by_h|} - \p_1^h \by_h\Big)
\times \frac{\p_2^h \by_h}{|\p_2^h \by_h|}\Big\|_{L^2_h(\o)} 
+ \Big\|\p_1^h \by_h
\times \Big( \frac{\p_2^h \by_h}{|\p_2^h \by_h|} - \p_2 \by_h \Big)\Big\|_{L^2_h(\o)} \\
&\le \big\||\p_1^h \by_h|-1 \big\|_{L^2_h(\o)} 
+ \big\|\p_1^h \by_h \big\|_{L^4_h(\o)} \big\||\p_2^h \by_h| -1 \big\|_{L^4_h(\o)}.
\end{split}\] 
By the approximate isometry property \cpam{and} {$\big|\p_j^h\by_h(\bz)\big|\ge1$
for all $\bz\in\mathcal{N}_h$ we obtain for $\by_h\in\cA_h^\delta$}
\begin{equation}\label{Lone-est}
\big\| |\p_j^h \by_h|-1\big\|_{L^1_h(\o)} 
\le \big\| |\p_j^h \by_h|^2 -1\big\|_{L^1_h(\o)} 
\le \delta.
\end{equation}
Moreover, since $\nabla_h\by_h$ is uniformly bounded in $H^1(\o)^3$ and
$\omega\subset\R^2$, we have by Sobolev embeddings for all $1\le q <\infty$
\begin{equation}\label{Lp-est}
\| \nabla_h \by_h \|_{L^q_h(\o)} \le c_q \| \nabla_h \by_h\|_{L^q(\o)}
\le c_q \big(\|\nabla_h \by_h \|_{L^2(\omega)} + \|\nabla\nabla_h\by_h\|_{L^2(\omega)}\big)
\le c_q, 
\end{equation}
\cpam{whence $\big\| |\p_j^h \by_h|-1\big\|_{L^q_h(\o)} \le c_q$.
Interpolating with discrete H\"older inequalities 
between this discrete $L^q_h$-estimate and the discrete
$L^1_h$-estimate in~\eqref{Lone-est}, we deduce that}
\[
\big\||\p_j^h \by_h| -1 \big\|_{L^p_h(\o)} \to 0, \qquad j=1,2
\]
for $p=2,4$ as \cpam{$\delta \to 0$.} This shows that $|I_h|\to0$.

The second term $II_h$ accounts for the effect of quadrature and
is \cpam{the same as} \eqref{quad}, whence 
\[
| II_h | \le c h \|\nabla\nabla_h\by_h\|_{L^2(\o)}^2 \|\nabla_h\by_h\|_{L^\infty(\o)}.
\]
Since $\by_h\in\cA_h^\delta$ is not an exact nodal isometry, we do not have direct
control of $\|\nabla_h\by_h\|_{L^\infty(\o)}$.
We invoke instead the \cpam{two-dimensional discrete Sobolev inequality
$\|\nabla_h\by_h\|_{L^\infty(\o)} \le c |\log h|^{1/2}
\|\nabla\nabla_h\by_h\|_{L^2(\o)}$~\cite[p.123]{BreSco08-book},} to infer that
\[
|II_h| \le c h |\log h|^{1/2} \|\nabla\nabla_h\by_h\|_{L^2(\o)}^3
\mathop{\longrightarrow}_{h\to0} 0.
\]
The last term $III_h$ is \cpam{the same as} \eqref{e:termIII} except that we
do not have $\by\in H^3(\omega)^3$. We split $III_h$ as follows:
\begin{align*}
III_h & = \Big(\big( \p_i \cI^1_h[\p_j^h\by_h]-\p_i\p_j\by \big)\cdot
\big[ \p_1\by\times\p_2\by \big], Z_{ij} \big)
\\
& + \Big( \p_i \cI^1_h[\p_j^h\by_h] \cdot \big\{ \big[\p_1^h\by_h\times\p_2^h\by_h]
- \big[ \p_1\by \times \p_2\by \big] \big\}, Z_{ij} \Big).
\end{align*}
We observe that $\p_i \cI_h^1[\p_j^h\by_h]\wto\p_i\p_j\by$ in
$L^2(\omega)^3$, whence the first term tends to $0$ as $h\to0$. In fact,
the uniform bound {\eqref{bound-Il2} on
$\nabla\cI_h^1[\nabla_h\by_h]$, in conjunction with
\eqref{apriori-H2},} implies the asserted weak \cpam{convergence,} and the limit is found via
\[
\big(\p_i\big(\cI^1_h[\p_j^h\by_h]-\p_j\by] \big),\psi\big)
= - \big(\cI^1_h[\p_j^h\by_h]-\p_j^h\by_h], \p_i\psi\big)
- \big(\p_j^h\by_h-\p_j\by , \p_i\psi \big) \mathop{\longrightarrow}_{h\to0} 0,
\]
which holds for every $\psi\in H^1_0(\omega)^3$
because
\[
\|\cI_h^1[\nabla_h\by_h] - \nabla_h\by_h\|_{L^2(\o)} \le c h 
\|\nabla\nabla_h\by_h\|_{L^2(\o)} \le c h 
\]
and $\nabla_h\by_h\to\nabla\by$ in $L^2(\o)$.
For the second term in $III_h$ we resort again to the uniform
$L^2$-bound on $\nabla \cI_h^1[\nabla_h\by_h]$ and write
\begin{align*}
\|\big[\p_1^h \by_h \times \p_2^h \by_h \big] 
- \big[\p_1 \by\times \p_2 \by \big]\|_{L^2(\omega)}
\le \|\nabla_h \by_h - \nabla \by \|_{L^4(\omega)}
\big( \|\nabla_h\by_h\|_{L^4(\omega)} + \|\nabla\by\|_{L^4(\omega)}\big).
\end{align*}
By compactness of the embedding $H^1(\o)\to L^4(\o)$,
we have  $\|\nabla_h \by_h - \nabla \by \|_{L^4(\omega)}\to0$ as
$h\to0$. Finally, using \eqref{Lp-est} for $q=4$ along with 
$\|\nabla\by\|_{L^\infty(\o)}\le c$ because $\by$ is an isometry, we
see that the preceding term tends to $0$ and thus conclude the proof.
\end{proof}

Theorem \ref{T:gamma-convergence} extends easily {to piecewise
constant approximations to $L^2$-data $Z$
and $\bef$ and to piecewise Lipschitz data over $\cT_h$;} we do
not carry out the details.
The following result is a consequence of standard abstract $\Gamma$-convergence theory
\cite{DalM93-book,Brai:14} {combined with Theorem \ref{T:gamma-convergence}
and Proposition~\ref{p:coercivity}.}
\begin{corollary}[{convergence of absolute minimizers}]\label{C:asb-min}
Let $\delta \to 0$ as $h\to 0$. Let $C>0$ be a constant independent of $h$
and $\{\by_h\}_h$ be a sequence of  almost absolute discrete minimizers
of $\tE_h$, namely
\begin{equation}\label{e:almost_absolute_min}
\tE_h[\by_h] \leq \inf_{\bw_h \in \cA_h^\delta} \tE_h[\bw_h] +
\epsilon_h \le C,
\end{equation}
where $\epsilon_h \to 0$ as $h\to 0$.
Then $\{\by_h\}_h$ is precompact in \cpam{$H^1(\o)^3$,} and every cluster 
point $\by$ of $\by_h$ is an absolute minimizer of \cpam{$\tE$}, namely
\begin{equation}\label{e:absolute_min_energy}
\cpam{\tE}[\by] = \inf_{{\bw \in \mathcal{A}}} \cpam{\tE}[\bw].
\end{equation}
Moreover, there exists a subsequence of $\{ \by_h \}_h$ (not relabeled) such that
\begin{equation}\label{e:absolute_min_displacement}
  \lim_{h\to 0} \| \by - \by_h \|_{H^1(\o)} = 0 \qquad \text{and}
  \qquad \lim_{h\to 0} \tE_h[\by_h] = \cpam{\tE}[\by].
\end{equation}
\end{corollary}

\begin{proof}
The uniform bound for the discrete energies and the coercivity property
of Proposition~\ref{p:coercivity} imply that the sequence
$\{\nabla_h\by_h\}_h$ is precompact in $L^2(\o)^{3\times2}$.
Due to the norm equivalence~\eqref{norm-equiv} and a Poincar\'e inequality
we have that $\{\by_h\}_h$ is bounded in $H^1(\o)^3$. Moreover, 
because of the estimate~\eqref{approx} the differences
$\nabla_h \by_h - \nabla \by_h$ converge strongly to zero in $L^2(\o)^{3\times 2}$ as $h\to 0$. 
Hence, there exists $\by\in H^1(\o)^3$ such that, up to the extraction of
a subsequence, we have
\[
\nabla_h \by_h, \, \nabla \by_h \to \nabla \by \quad \text{ in } L^2(\o).
\]
The lower bound assertion of Theorem~\ref{T:gamma-convergence} implies that
$\by\in \cA$ and 
\begin{equation}\label{est:liminf_min}
\cpam{\tE}[\by] \le \liminf_{h \to 0} \tE_h[\by_h].
\end{equation}
It remains to show that $\by$ is a global minimizer of \cpam{$\tE$}. To prove this,
let $\eta>0$ be arbitrary and $\bz\in \cA$ such that 
\[
\cpam{\tE}[\bz] \le \inf_{\bw\in \cA} \cpam{\tE}[\bw] + \eta/2.
\]
The attainment property stated in Theorem~\ref{T:gamma-convergence} implies
that there exist $h>0$ and $\bz_h \in \cA_h^\d$ so that 
\[
\tE_h[\bz_h] \le \cpam{\tE}[\bz] + \eta/2. 
\]
On combining the previous two estimates and incorporating the fact that
$\by_h$ is a minimizer for $\tE_h$ in $\cA_h^\d$ up to the value $\veps_h$, 
we have that 
\[
\tE_h[\by_h] \le \tE_h[\bz_h] + \veps_h \le \cpam{\tE}[\bz] + \eta/2 + \veps_h 
\le \inf_{\bw\in \cA} \cpam{\tE}[\bw] + \eta + \veps_h.
\]
This together with~\eqref{est:liminf_min} and the arbitrariness of 
$\eta>0$ prove \eqref{e:absolute_min_displacement}. 
\end{proof}

\begin{remark}[local minimizers]\label{r:local_min}
\rm Statements about almost local minimizers of $\tE_h$ are not
available in general. However, if $\tE$ has an isolated local
minimizer $\by$, then there exist local minimizers
$\{\by_h\}_h$ of $\tE_h$ converging to $\by$ provided $h$ is
sufficiently small \cite[Theorem 5.1]{Brai:14}. We defer the
discussion of almost local discrete minimizers of $\tE_h$ to Section~\ref{S:experiments}.
\end{remark}

\section{Fully Discrete Gradient Flow} \label{S:gradient-flow}

{
Corollary~\ref{C:asb-min} guarantees that every accumulation point of almost absolute minimizers of $\{\tE_h\}_h$ is an absolute minimizer of $E$. 
We {introduce} and study in this section a practical gradient flow algorithm to minimize $\tE_h$ on $\cA_h^\delta$ for $h>0$ and where $\delta$ is proportional to the gradient flow pseudo-time parameter.
Our fully discrete gradient flow {gives rise to} an energy decreasing 
iterative scheme that converges to stationary points {satisfying}
the isometry constraint up to a small error.
However, like every gradient descent method, whether the algorithm
reaches an almost absolute minimizer, a saddle point, or a local
minimizer is not {possible to discern. We discuss this further in
Section~\ref{S:experiments}.}
}

\begin{algorithm}[discrete $H^2$-gradient flow]\label{alg:grad_flow_discr}
Let $\tau >0$ and set $k=0$. Choose $\by_h^0 \in \cA_h^0$. \\
(1) Compute $\by_h^{k+1}\in \by_h^k + \cF_h\big[\by^k_h\big]$ which is minimal for 
the functionals
\[
\by_h \mapsto \frac{1}{2\tau} \|\nabla \nabla_h (\by_h - \by_h^k)\|_{L^2(\o)}^2
+ \tE_h[\by_h]
\]
in the set of all $\by_h \in \by_h^k + \cF_h\big[\by_h^k\big]$.\\
(2) increase $k\to k+1$ and continue with~(1).
\end{algorithm}

Every step of the gradient flow requires solving a {\it nonconvex} minimization
problem. Since the primary variables of interest are the discrete gradients
\[
\Phi_h := \nabla_h \by_h^{k+1}, \quad \tPhi_h:=\nabla_h \by_h^k, \quad
\Psi_h := \nabla_h \bw_h,
\]
with $\bw_h\in \cF_h\big[\by_h^k\big]$, we let their columns be
 $\Phi_{h,j}$, $\tPhi_{h,j}$, $\Psi_{h,j}$
for $j=1,2$, and write the corresponding Euler--Lagrange equations 
as follows:
\begin{equation}\label{fully-discrete-gradient}
\begin{aligned}
\frac{1}{\tau} & (\nabla [\Phi_h-\tPhi_h],\nabla \Psi_h) + (\nabla \Phi_h,\nabla \Psi_h) 
+\sum_{i,j=1}^2 \Big( \p_i \cI^1_h[\Psi_{h,j}] \cdot  
\Big[\frac{\Phi_{h,1}}{|\Phi_{h,1}|} \times \frac{\Phi_{h,2}}{|\Phi_{h,2}|}\Big], Z_{ij} \Big)_h \\
&+\sum_{i,j=1}^2 \Big(\p_i \cI^1_h[\Phi_{h,j}] \cdot
\Big[\big(P_{\Phi_{h,1}} \Psi_{h,1}\big) \times \frac{\Phi_{h,2}}{|\Phi_{h,2}|}+
\frac{\Phi_{h,1}}{|\Phi_{h,1}|} \times \big(P_{\Phi_{h,2}} \Psi_{h,2}\big)\Big], Z_{ij} \Big)_h
= (\bef,\bw_h)_h
\end{aligned}
\end{equation}
for all $\bw_h\in \cF_h\big[\by_h^k\big]$. Hereafter, to have a
simple and compact notation, we \cpam{let $P_\ba$ be the operator}
\[
P_\ba := \frac{1}{|\ba|} \Big(\Id_3 - \frac{\ba^\transp \ba }{|\ba|^2} \Big),
\]
\cpam{for any given $\ba\in\R^3$ and observe that $|P_\ba|\le 1$
provided $|\ba|\ge1$.}
Notice that we omit writing the time steps $k$ and $k+1$ in 
\eqref{fully-discrete-gradient}. Existence of a \cpam{locally} unique solution to
\eqref{fully-discrete-gradient} follows from \cpam{a local} contraction property of 
the fixed-point iteration defined in the next algorithm, in which we write
$\by_h^\ell$ for $\by_h^{k,\ell}$.

\begin{algorithm}[fixed-point iteration]\label{alg:fixed_point}
Let {$\tby_h \in \cA_h^{\infty}$}, define $\by_h^0=\tby_h$,
and set $\ell=0$. \\
(1) Compute $\Phi_h^{\ell+1}:= \nabla_h \by_h^{\ell+1}$ with 
$\by_h^{\ell+1} \in \tby_h + \cF_h\big[\tby_h\big]$ such that
\begin{equation}\label{iteration}
\begin{aligned}
\frac{1}{\tau} (\nabla [\Phi_h^{\ell+1}&-\tPhi_h],\nabla \Psi_h) 
+ (\nabla \Phi_h^{\ell+1},\nabla \Psi_h)  
= - \sum_{i,j=1}^2 \Big(  \p_i \cI^1_h[\Psi_{h,j}] \cdot  
\Big[\frac{\Phi_{h,1}^\ell}{|\Phi_{h,1}^\ell|} 
  \times \frac{\Phi_{h,2}^\ell}{|\Phi_{h,2}^\ell|}\Big],Z_{ij} \Big)_h \\
& - \sum_{i,j=1}^2 \Big( \cpam{\p_i\cI^1_h[\Phi_{h,j}^\ell]} \cdot
\Big[P_{\Phi_{h,1}^\ell} \Psi_{h,1} \times \frac{\Phi_{h,2}^\ell}{|\Phi_{h,2}^\ell|} +
\frac{\Phi_{h,1}^\ell}{|\Phi_{h,1}^\ell|} \times P_{\Phi_{h,2}^\ell} \Psi_{h,2}\Big],Z_{ij}\Big)_h
+ (\bef,\bw_h)_h 
\end{aligned}
\end{equation}
for all $\bw_h\in \cF_h[\tby_h]$ with $\Psi_h:=\nabla_h \bw_h$.\\
(2) increase $\ell \to \ell+1$ and continue with~(1).
\end{algorithm}

\cpam{Note that \eqref{iteration} is a linear system for $\Phi_h^{\ell+1}$.}
Under a moderate condition on the step size $\tau>0$ 
\cpam{the iterates are uniformly bounded and the map
$\Phi_h^{\ell+1}\mapsto\Phi_h^\ell$ is a
contraction in a suitable $H^1$-ball}. In particular, the 
limiting discrete Euler--Lagrange equations
\eqref{fully-discrete-gradient} then \cpam{admit a locally unique solution}. This, and
related properties, are discussed in the next two propositions.

\begin{proposition}[\cpam{local} contraction property]\label{prop:fp_conds}
\cpam{Let the mappings $\bef, Z$ be elementwise continuous, let
$\widetilde\Phi_h = \nabla_h \widetilde\by_h$, and set
$\widetilde{C} := \max\{1,\|\nabla \tPhi_h \|_{L^2(\o)}\}$. If
\begin{equation*}
  \mathcal{B}_h := \big\{ \by_h\in\mathcal{A}_h^\infty: \quad
  \|\nabla\nabla_h\by_h\|_{L^2(\omega)} \le 2 \widetilde{C}
  \big \} ,
\end{equation*}
then the nonlinear map $\Phi_h^{\ell} \mapsto \Phi_h^{\ell+1}$ in \eqref{iteration} is
well-defined from $\mathcal{B}_h$ into itself and is a contraction
with constant $\frac12$ provided that $\tau \le C_0$, with $C_0>0$
depending explicitly on 
$\widetilde{C}$, a Poincar\'e constant $c_P\ge1$ of $\o$,
$\|\bef\|_{L^\infty(\o)}$ and $\|Z\|_{L^\infty(\o)}$. Consequently, if
Algorithm~\ref{alg:fixed_point} is initialized with the $k$-th iterate of Algorithm
\ref{alg:grad_flow_discr}, i.e. $\widetilde\by_h = \by_h^k$, then the solution $\Phi_h^\ell$ of
Algorithm \ref{alg:fixed_point} converges to the unique solution of the 
Euler-Lagrange equation \eqref{fully-discrete-gradient} within
$\mathcal{B}_h$.
}
\end{proposition}

\begin{proof} We proceed in three steps and simplify the
      notation upon writing $\|\cdot\|=\|\cdot\|_{L^2(\o)}$.

\smallskip
(i) {\it Existence and isometry relation}:
\cpam{
Given $\by_h^\ell\in\mathcal{B}_h$, and in particular $\by_h^\ell\in\mathcal{A}_h^\infty$,
we see that $[\Phi_h^\ell(\bz)]^T \Phi_h^\ell(\bz) \ge I_2$ for all $\bz\in\cN_h$.
We thus have} $|\p_j^h \by_h^\ell|\ge 1$ for $j=1,2$, so
that the right-hand side of \eqref{iteration}
is well-defined and the Lax--Milgram lemma gives the existence 
of a unique solution $\by_h^{\ell+1} \in \tby_h +  \cF_h\big[\tby_h\big]$.
Due to the definition of $\cF_h\big[\tby_h\big]$ we infer that
\[
\big[\Phi_h^{\ell+1}(\bz) -\tPhi_h(\bz)\big]^\transp  \tPhi_h (\bz) 
+ \tPhi_h(\bz)^\transp \big[\Phi_h^{\ell+1}(\bz) -\tPhi_h(\bz)\big] = 0
\]
for all $\bz\in \cN_h$. This implies $\by_h^{\ell+1}\in \cA_h^\infty$, namely
\[
[\Phi_h^{\ell+1}(\bz)]^\transp  \Phi_h^{\ell+1} (\bz) 
= [\tPhi_h(\bz)]^\transp  \tPhi_h (\bz) 
+ \big[\Phi_h^{\ell+1}(\bz) -\tPhi_h(\bz)\big]^\transp 
\big[\Phi_h^{\ell+1}(\bz) -\tPhi_h(\bz)\big]
\ge I_2
\]
for all $\bz\in \cN_h$, because $\tby_h\in\cA_h^\infty$.

\smallskip
(ii) {\it Uniform bound}: 
We next show that the iterates
\cpam{satisfy $\|\nabla\Phi_h^\ell\|\le2\widetilde{C}$.} For
this, we choose 
\[
\Psi_h = \Phi_h^{\ell+1}-\tPhi_h =\nabla_h \big[\by_h^{\ell+1}-\tby_h\big]
\]
in \eqref{iteration}. This leads to
\begin{align*}
 \frac{1}{\tau} \|\nabla (\Phi_h^{\ell+1} &-\tPhi_h)\|^2 
+ \frac12 \|\nabla (\Phi_h^{\ell+1}-\tPhi_h)\|^2 
+ \frac12 \|\nabla \Phi_h^{\ell+1}\|^2 - \frac12 \|\nabla \tPhi_h\|^2  \\
& = (\bef,\by_h^{\ell+1}-\tby_h)_h
- \sum_{i,j=1}^2 \Big( \p_i \cI^1_h[\Phi_{h,j}^{\ell+1}-\tPhi_{h,j}] \cdot  
\Big[\frac{\Phi_{h,1}^\ell}{|\Phi_{h,1}^\ell|} 
  \times \frac{\Phi_{h,2}^\ell}{|\Phi_{h,2}^\ell|}\Big],Z_{ij}\Big)_h \\
& - \sum_{i,j=1}^2 \Big(\p_i \cI^1_h[\Phi_{h,j}^\ell] \cdot
\Big[P_{\Phi_{h,1}^\ell}(\Phi_{h,1}^{\ell+1}-\tPhi_{h,1}) \times \frac{\Phi_{h,2}^\ell}{|\Phi_{h,2}^\ell|} +
\frac{\Phi_{h,1}^\ell}{|\Phi_{h,1}^\ell|} \times P_{\Phi_{h,2}^\ell}(\Phi_{h,2}^{\ell+1}-\tPhi_{h,2})\Big] ,Z_{ij}\Big)_h.
\end{align*}
Since \cpam{$|P_{\Phi^\ell_h(\bz)} |\le 1$} for all $\bz\in\cN_h$, we
\cpam{deduce
\begin{equation*}
\frac{1}{\tau} \|\nabla (\Phi_h^{\ell+1} - \tPhi_h)\|^2 \le 
\frac12 \|\nabla \tPhi_h\|^2 
+  c(f,Z) \big(\|\by_h^{\ell+1}-\widetilde\by_h\| 
+ \|\nabla (\Phi_h^{\ell+1}-\tPhi_h)\|  + \|\nabla \Phi_h^\ell\|  
\| \Phi_h^{\ell+1}-\tPhi_h \| \big),
\end{equation*}
with $c(f,Z)>0$ only depending on $\bef$ and $Z$.
We now apply the Poincar\'e inequality to $\|\by_h^{\ell+1}-\tilde\by_h\|$ in
conjunction with \eqref{norm-equiv}, namely 
$\|\by_h^{\ell+1}-\tilde\by_h\| \le c_P \|\Phi_h^{\ell+1} - \tPhi_h\|$,
where we let $c_P$ be the product of the Poincar\'e constant of $\o$
and the constant $c_1\ge1$ in \eqref{norm-equiv} and take it to be $c_P\ge1$.
Applying the Poincar\'e inequality again, this time to
$\Phi_h^{\ell+1} - \tPhi_h$, we obtain
\begin{equation*}
\frac{1}{\tau} \|\nabla (\Phi_h^{\ell+1} - \tPhi_h)\|^2 
\le \frac12 \|\nabla \tPhi_h\|^2 
+ c(f,Z) c_P^2 \big(2 + \|\nabla \Phi_h^\ell\|  \big) 
\|\nabla (\Phi_h^{\ell+1}-\tPhi_h)\|.
\end{equation*}
To prove that $\|\nabla\Phi_h^{\ell+1}\|\le2\widetilde{C}$ we assume
by induction that $\|\nabla\Phi_h^\ell\|\le2\widetilde{C}$, which is
valid for $\ell=0$.
Using $\|\nabla \tPhi_h\|\le \widetilde{C}$ yields
\begin{equation*}
\|\nabla (\Phi_h^{\ell+1}-\tPhi_h)\|^2 
\le \tau \widetilde{C}^2 + 4\tau^2 c(f,Z)^2 c_P^4 \big(1 + \widetilde{C}\big)^2.
\end{equation*}
Since $\widetilde{C}\ge1$, choosing
\[
\tau \le C_1 := \min\Big\{\frac12,\frac{1}{16 c(f,Z)^2 c_P^4} \Big\}
  \le \frac{1}{4c(f,Z)^2 c_P^4} \frac{\widetilde C^2}{(1+\widetilde C)^2}
\]
implies $\|\nabla (\Phi_h^{\ell+1}-\tPhi_h)\|^2
\le 2\tau\widetilde{C}^2\le \widetilde{C}^2$,
whence $\|\nabla\Phi_h^{\ell+1}\|\le2\widetilde{C}$ as asserted.
}

\smallskip
(iii) {\it Contraction property}:
It remains to show that \cpam{the map $\Phi_h^\ell\mapsto\Phi_h^{\ell+1}$
given in \eqref{iteration} is a contraction on $\mathcal{B}_h$.} 
For this, we subtract the equations that define $\Phi_h^{\ell+1}$ and
$\Phi_h^\ell$ in \eqref{iteration} and verify that 
\[\begin{split}
\frac{1}{\tau}   (\nabla (\Phi_h^{\ell+1}-\Phi_h^\ell),\nabla \Psi_h) 
& + (\nabla (\Phi_h^{\ell+1}-\Phi_h^\ell),\nabla \Psi_h) \\
&= - \sum_{i,j=1}^2 \Big(  \p_i \cI^1_h[\Psi_{h,j}] \cdot  
\Big[\frac{\Phi_{h,1}^\ell}{|\Phi_{h,1}^\ell|} 
  \times \frac{\Phi_{h,2}^\ell}{|\Phi_{h,2}^\ell|}\Big], Z_{ij} \Big)_h \\
& - \sum_{i,j=1}^2 \Big(\p_i \cI^1_h[\Phi_{h,j}^\ell] \cdot
\Big[P_{\Phi_{h,1}^\ell} \Psi_{h,1} \times \frac{\Phi_{h,2}^\ell}{|\Phi_{h,2}^\ell|} +
\frac{\Phi_{h,1}^\ell}{|\Phi_{h,1}^\ell|} \times P_{\Phi_{h,2}^\ell} \Psi_{h,2}\Big], Z_{ij}\Big)_h \\
& + \sum_{i,j=1}^2 \Big( \p_i \cI^1_h[\Psi_{h,j}] \cdot  
\Big[\frac{\Phi_{h,1}^{\ell-1}}{|\Phi_{h,1}^{\ell-1}|} 
  \times \frac{\Phi_{h,2}^{\ell-1}}{|\Phi_{h,2}^{\ell-1}|}\Big], Z_{ij}\Big)_h \\
& + \sum_{i,j=1}^2 \Big( \p_i \cI^1_h[\Phi_{h,j}^{\ell-1}] \cdot
\Big[P_{\Phi_{h,1}^{\ell-1}} \Psi_{h,1} \times \frac{\Phi_{h,2}^{\ell-1}}{|\Phi_{h,2}^{\ell-1}|} +
\frac{\Phi_{h,1}^{\ell-1}}{|\Phi_{h,1}^{\ell-1}|} \times
P_{\Phi_{h,2}^{\ell-1}} \Psi_{h,2}\Big], Z_{ij}\Big)_h .
\end{split}\]
\cpam{Bounding separately the sum of the first and third terms and
that of the second and
fourth terms, and using} the admissible  
choice $\Psi_h = \Phi^{\ell+1}_h-\Phi^\ell_h$, we deduce 
the estimate 
\[
\frac{1}{\tau} \|\nabla (\Phi_h^{\ell+1}-\Phi_h^\ell)\|^2
\le \frac12 C_2^{-1} \|\nabla (\Phi_h^{\ell+1}-\Phi_h^\ell)\| 
\|\nabla (\Phi_h^{\ell-1}-\Phi_h^\ell)\|,
\]
where we \cpam{have used $\|\nabla\Phi_h^\ell\| \le 2 \widetilde{C}$
shown in step (ii),}
that $|\Phi_{h,j}^\ell(\bz)|\ge 1$ for all $\ell\ge 0$, $j=1,2$, and $z\in \cN_h$, 
and that for $\ba,\bb\in \R^3$ with $|\ba|,|\bb|\ge 1$ the following estimates
hold
\[
\Big|\frac{\ba}{|\ba|} - \frac{\bb}{|\bb|}\Big| \le |\ba-\bb|, \quad 
\big|P_\ba - P_\bb\big| = 3 |\ba-\bb|.
\]
This implies the 
asserted contraction property \cpam{with constant $\frac12$ in
$\mathcal{B}_h$} for $\tau \le C_0:=\min\{C_1,C_2\}$.
\cpam{Finally, if $\widetilde\by_h = \by_h^k$
is the $k$-th iterate of Algorithm \ref{alg:grad_flow_discr}, then any
fixed point of \eqref{iteration} is a solution of
\eqref{fully-discrete-gradient} and conversely. This implies
uniqueness of \eqref{fully-discrete-gradient} within $\mathcal{B}_h$.
}
\end{proof}

\begin{remark}[time step]\label{R:timestep}
\rm
If $\bef=0$, then the preceding proof shows that $\tau$ has to be 
sufficiently small \cpam{so that 
$
\tau \le \big(4 c(Z)^2 c_P^2\big)^{-1},
$
where $c(Z) = \|Z\|_{L^\infty(\omega)}$} and $c_P\ge1$ is the Poincar\'e
constant of $\omega$ with vanishing Dirichlet condition on $\p_D\o$.
\cpam{Due to a repeated application of the Poincar\'e inequality,
the case $\bef\neq 0$ requires a stringent condition on $\tau$.}
\end{remark}

The following proposition shows that the discrete $H^2$ 
gradient flow of Algorithm \ref{alg:grad_flow_discr}
is energy decreasing and becomes stationary,
its iterates are uniformly 
bounded, and the violation of the isometry constraint is
controlled by the pseudo-timestep size.

\begin{proposition}[properties of iterates]\label{prop:grad_flow}
Let $Z$ and $\bef$ be piecewise continuous over the partition
  $\cT_h$. Let
$\{\by_h^k\}_{k=0}^\infty \subset  {\cA_h^\infty}$ be iterates of
Algorithm~\ref{alg:grad_flow_discr}. We then have that for all $k\ge 0$
\begin{equation}\label{energy-decrease}
\tE_h[\by_h^{k+1}] 
+ \frac{1}{2\tau}\sum_{\ell=0}^k  \|\nabla\nabla_h (\by_h^{\ell+1}-\by_h^\ell)\|^2
\le \tE_h[\by_h^0],
\end{equation}
and in particular
\begin{equation}\label{H2-bound}
\|\nabla \nabla_h \by_h^k\|\le \widetilde{C}
\end{equation}
for \cpam{a constant $\widetilde{C}>0$ depending on $\by_h^0, f$ and
$Z$, but independent of $k$.}  In addition, 
if $\big[\nabla_h \by_h^0(\bz)\big]^\transp \big[\nabla_h \by_h^0(\bz)\big] = I_2$ for
all $\bz\in \cN_h$ then $\by_h^{k+1} \in {\cA_h^{c_0\tau}}$, i.e.
\begin{equation}\label{isometry-defect}
\big\|[\nabla_h \by_h^{k+1}]^\transp [\nabla_h \by_h^{k+1}]- \Id_2 \big\|_{L^1_h(\o)}
\le  c_0 \tau \qquad\forall k\ge0,
\end{equation}
where $c_0$ depends \cpam{only on $\by_h^0$.} 
\end{proposition}

\begin{proof} The proof splits into three steps.

\smallskip
(i) {\it Energy decay}: This is a direct consequence of the minimizing
properties of the iterates, i.e., 
\[
\tE_h[\by_h^{k+1}] + \frac{1}{2\tau} \|\nabla\nabla_h(\by_h^{k+1}-\by_h^k)\|^2_{L^2(\o)} 
\le\tE_h[\by_h^k].
\]

(ii) {\it Coercivity}:
For every \cpam{$\by_h\in \cA_h^\infty$} we have $|\p_j^h \by_h(\bz)|\ge 1$ for $j=1,2$ 
and all $\bz\in \cN_h$. \cpam{This, in conjunction 
with~\eqref{coercivity} and \eqref{bound-Il2}, yields}
\[
\tE_h[\by_h] \ge \frac{1}{4} \|\nabla \nabla_h \by_h\|^2_{L^2(\o)}
- c \big(\|Z\|^2_{L^\infty(\o)} + \|\bef\|^2_{L^\infty(\o)} \big).
\]
Since $\tE_h[\by_h^k]\le\tE_h[\by_h^0]$ \cpam{from (i)},
this implies the asserted bound \eqref{H2-bound} of the iterates.

\smallskip
(iii) {\it Isometry violation}:
Abbreviating $\Phi_h^{k+1}:=\nabla_h \by_h^{k+1}$ and $\Phi_h^k := \nabla_h \by_h^k$,
and noting that $\by_h^{k+1}-\by_h^k\in \cF[\by_h^k]$, we have
\cpam{for all $\bz\in \cN_h$}
\[
\big[\Phi_h^{k+1}(z)-\Phi_h^k(z)\big]^\transp \Phi_h^k(z) 
+ \big[\Phi_h^k(z)\big]^\transp \big[\Phi_h^{k+1}(z)-\Phi_h^k(z)\big] = 0,
\]
whence 
\[
[\Phi_h^{k+1}(z)]^\transp [\Phi_h^{k+1}(z)] 
= [\Phi_h^k(z)]^\transp [\Phi_h^k(z)] 
+ [(\Phi_h^{k+1}-\Phi_h^k) (z)]^\transp [(\Phi_h^{k+1}-\Phi_h^k)(z)].
\]
A repeated application of this identity along with
$[\nabla_h \by_h^0(z)]^\transp [\nabla_h \by_h^0(z)]=\Id_2$ for all $z\in \cN_h$
yields 
\[
\big\|[\Phi_h^{k+1}]^\transp [\Phi_h^{k+1}] - \Id_2 \big\|_{L^1_h(\o)}
\le  c \sum_{\ell=0}^k \|\nabla_h(\by_h^{\ell+1}-\by_h^\ell)\|^2. 
\]
With the help of a Poincar\'e inequality for
$\nabla_h(\by_h^{\ell+1}-\by_h^\ell)$ and \eqref{energy-decrease},
we deduce \eqref{isometry-defect}.
\end{proof}

We end this section by pointing out that the stationary state $\by_h^\infty$ reached
by the gradient flow might not be a discrete (almost) absolute minimizer.
However, assuming that $\by_h^\infty$ is an (almost) absolute minimizer, then Corollary~\ref{C:asb-min} guarantees that the accumulation points when $h\to 0$ are absolute minimizers of the exact energy \cpam{$\tE$.}
 
\section{Numerical Experiments: 
Performance and Model Exploration}\label{S:experiments}

\cpam{Algorithms~\ref{alg:grad_flow_discr} and~\ref{alg:fixed_point} are} implemented using 
the {\tt deal.II} library \cite{BHK:07}. \cpam{All systems of linear equations
arising in the iteration of Algorithm~\ref{alg:fixed_point} where solved directly 
using {\tt UMFPACK} \cite{davis2007umfpack} leaving the discussion on developing efficient solvers open}.
The resulting deformations are visualized with {\tt paraview} \cite{paraview}.
Note that only the displacement degrees of freedom at the vertices of
$\cT_h$ are used for this purpose and the plates are thus displayed as 
continuous piecewise bi-linear elements.
In addition, we say that a plate has reached numerically {an equilibrium state
parametrized by $\by_h^{k+1}$} when
\begin{equation}\label{e:stop_criteria_num}
\frac{|\tE_h[\by_h^{k+1}]-\tE_h[\by_h^k]|}{\tau} \leq 10^{-6},
\end{equation}
where $\tau$ is the gradient flow timestep.
We set $\by_h^\infty:=\by_h^{K+1}$, where $K$ is the smallest 
$k$ that satisfies the above constraint.
{Each pseudo-time iteration $k$ of the gradient flow consists of a 
fixed-point iteration (Algorithm~\ref{alg:fixed_point}).
This inner loop stops when the $(\ell+1)$-th inner iterate satisfies 
$$
\| \nabla \nabla_h (\by_h^{k,\ell+1} - \by_h^{k,\ell})\| \leq \delta_{\text{stop}}
$$
for some $\delta_{\text{stop}}>0$.
It turns out that the number of subiterations experienced in practice is low 
(see for instance Figure~\ref{f:energy-decay}).}
\cpam{For ease of computation} the discrete energies are approximated according 
to 
$$
\tE_h[\by_h] \approx \frac 1 2 \int_\omega |H_h{+Z}|^2,
$$
where $H_h$ is the approximation of the second fundamental form given by
\[
H_{h,i,j}:= \big((\partial_1^h \by_h)_1 \times (\partial_2^h \by_h)_2\big) 
\cdot  \partial_i\partial_j^h \by_h,
\]
without normalization of the vectors $\partial_i^h \by_h$, $i=1,2$,
and $Z$ is given and symmetric. The absence of space dependence in $Z$ 
corresponds to homogeneous materials.

The purpose of the following numerical experiments is twofold. We
first document the performance of Algorithms \ref{alg:grad_flow_discr}
and \ref{alg:fixed_point} by investigating
their behavior for decreasing values of meshsize $h$ and
pseudo-timestep $\tau$ and examining the violation of the isometry
constraint. We also study \cpam{various qualitative properties} of the 
\cpam{dimensionally} reduced model, the existence of local
discrete minimizers other than cylinders, which are \cpam{for appropriate
data and boundary conditions} global minimizers
according to \cite{Schm:07a}, as well as the pseudo-evolution process
(sometimes exhibiting self-intersections).

\cpam{We remark that our underlying nonlinear mathematical 
model allows for the description of large deformations which may be 
nonunique and cannot be expected to admit high regularity. Therefore, a 
meaningful error analysis seems out of reach. For small displacements 
the model leads to a linear bending problem for which best-approximation 
and interpolation results imply convergence rates. An experimental 
convergence analysis for the case of a simple exact solution 
reported below indicates a linear convergence rate.} We document our quantitative
findings in Tables \ref{t:clamped_energy_1}, \ref{t:clamped_isometry_1},
\cpam{and \ref{table:exp_conv_text}} and use the symbol N/A to indicate special 
combinations {of $h$ and $\tau$} for which we did not perform computations 
\cpam{as these appeared to be irrelevant for the discussion and in some 
cases computationally expensive.}

Bilayer bending has technological applications in design and fabrication 
of micro-switches and micro-grippers as well as nano-tubes 
\cite{BALG:10,JSI:00,KLPL:05,SchEb:01,SIL:95}. In these cases it 
is essential that the bilayer plate undergoes a complete folding to a 
cylinder without exhibiting {\it dog-ears} or a {\it corkscrew} shape, 
which may affect or impede the complete folding \cite{ADMA:ADMA19930050905}. 
Better understanding and control of this phenomenon is what motivated this work. 
We describe below several equilibrium configurations other than cylinders.

\subsection{Benchmark}\label{ss:benchmark}

We display in Figure~\ref{f:benchmark}
the pseudo-evolution of a bilayer plate $\omega = (-5,5) \times (-2,2)$, 
clamped on the left-side $\partial_D\o = \{x=-5\} \times [-2,2]$, i.e., 
$$
\by = 0, \qquad \nabla \by = I_{3\times 2} \qquad \text{on} \quad \p_D \o,
$$
with a spontaneous  curvature {$Z=-I_{2}$}.
The finite element partition consists of 5 uniform refinements of the
rectangle $\o$. The parameters \cpam{for Algorithms~\ref{alg:grad_flow_discr} 
and~\ref{alg:fixed_point}} are the pseudo-timestep $\tau=0.005$ and the 
sub-iteration stopping tolerance $\delta_{\text{stop}}=10^{-4}$.
The discrete equilibrium state is a cylinder. 

\begin{figure}[ht!]
\centerline{\includegraphics[width=0.75\textwidth]{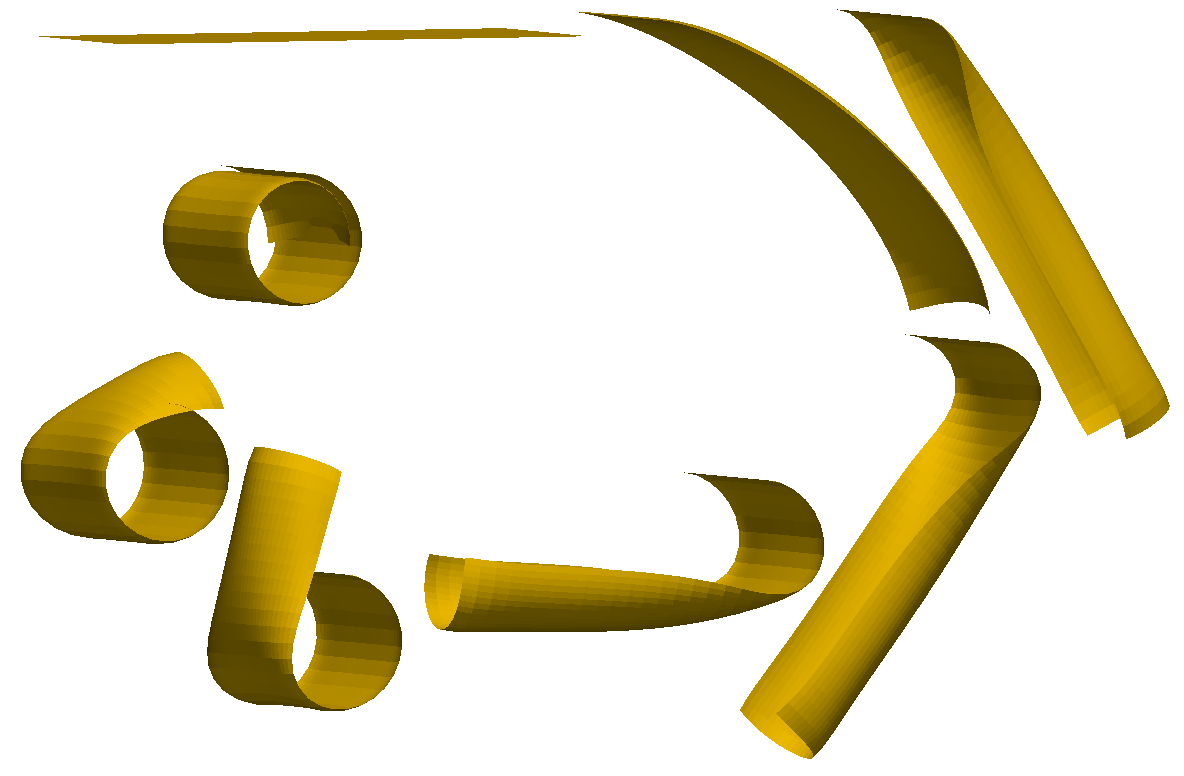}}
\caption{\small 
Pseudo-evolution (clockwise) towards the equilibrium of a
  clamped rectangular plate with spontaneous curvature {$Z=-I_2$}. The
  bilayer plate is depicted (clockwise) {for 0.0, 0.1, 2.4, 60.0, 100.0,
  130.0, 135.0, 207.3 times $10^3$ iterations} of Algorithm~\ref{alg:grad_flow_discr}.
The bilayer plate reaches a cylindrical shape asymptotically (limit of
discrete gradient flow). This is an absolute minimizer.} 
\label{f:benchmark}
\end{figure}

\subsection{Relaxation Process}\label{ss:relaxation}

We consider the clamped plate $\o$ described in
Section~\ref{ss:benchmark} for different spontaneous curvatures {$Z$}
as well as different discretization parameters $h,\tau$. 
The subiterations stopping tolerance for Algorithm
\ref{alg:fixed_point} is $\delta_{\text{stop}}=10^{-3}$.
We examine the relaxation process towards equilibrium for two
different spontaneous curvatures, namely {$Z=-I_2,\,-5 I_2$}.
According to \cite{Schm:07a}, the absolute energy minimizers are
cylinders of height 4 (length of the clamped side) and of radius
$1$ and $1/5$ (reciprocal of the eigenvalues of {$Z$}) with an energy of
$20$ and $500$, respectively. 

Table \ref{t:clamped_energy_1} documents the 
influence of meshsize $h$ and
pseudo-timestep $\tau$ on the equilibrium shape and corresponding
energy. The cases {$Z=-I_2$ and $Z=-5I_2$} are strikingly different. 
For {$Z=-I_2$} the equilibrium shape is
a cylinder (absolute minimizer) provided
$\tau \leq C_0 h$ for a sufficiently small constant $C_0$;
see Figure~\ref{f:equilibrium} (left for $\tau=0.0025$ and middle for
$\tau=0.005$). For {$Z=-5I_2$}, and regardless of the size of $\tau$,
the plate never reaches a
cylinder but other equilibrium configurations (local minimizers) 
with much higher energy than $500$; 
see Figure \ref{f:equilibrium} (right). A plausible explanation is
that the relatively large spontaneous curvature 
in the direction of the clamped side favors bending in such a
direction, thereby creating a geometric obstruction to reaching
a cylindrical shape.

Table \ref{t:clamped_energy_1} also
provides information about the threshold of $\tau$ needed for convergence of the
sub-iterations of Algorithm \ref{alg:fixed_point}. This value, being sensitive to 
$\|Z\|_{L^\infty(\o)}$, is more stringent
for {$Z=-5I_2$}. In fact, such iterations fail to converge for
$\tau=0.02$ and {$Z=-5I_2$, whereas for $Z=-I_2$} give rise to a local
minimizer. This is consistent with Remark~\ref{r:local_min}, which establishes 
the pessimistic thresholds $\tau_0=2.5 \cdot10^{-3}$ for {$Z=-I_2$} and 
$\tau_0=10^{-4}$ for {$Z=-5I_2$} if we {consider} a Poincar\'e constant
$c_P=10$.

\cpam{In addition, we note that the case $\tau = 0.00125$ on the mesh \#6 
and with $Z=-I_2$ yields a cylindrical equilibrium shape with energy of $17.2$ 
(not reported in Table~\ref{t:clamped_energy_1}). 
This, in conjunction with the energies when $\tau=0.005$ on the mesh 
resulting from four refinements and $\tau=0.0025$ on the mesh 
resulting from five refinements, illustrates that 
$
\tE_h[\by_h^\infty] \to E[\by] = 20
$
as $h \to 0$, $h = c \tau$ and $c$ is sufficiently small to obtain cylindrical shapes. 
This is in accordance with Corollary~\ref{C:asb-min}.}

\begin{table}
\begin{tabular}{c||cccc||cccc}
       &  \multicolumn{4}{c||}{ {$Z=-I_2$} } & \multicolumn{4}{c}{
{$Z=-5I_2$} }\\
\diaghead{\theadfont Refinement}{$\tau$}{Mesh} &  \#4  & \#5 & \#6 & \#7 &  \#4
& \#5 & \#6 & \#7 \\
\hline\hline
0.02 & 19.781n  & 20.351n & N/A  &  N/A &  NoC & NoC & N/A  &  N/A \\
0.01 &  19.335n & 20.157n & 20.576n & 19.590n & 575.372 & 521.297 &
536.036 & 586.839 \\
0.005 & 15.961y & 16.554y& 20.343n & N/A  & 567.886 & 519.599& 534.365
& N/A  \\
0.0025 & 15.765y & 16.395y & 17.304y &  18.062y & 581.405 & 518.897 &
N/A &  N/A
\end{tabular}
\medskip
\caption{\small
Equilibrium energies $\tE_h[\by_h^\infty]$ for spontaneous
curvatures {$Z=-I_2, -5I_2$} and different meshsizes $h$ and pseudo-timesteps $\tau$.
The meshes correspond to 4, 5, 6, and 7 uniform refinements of the
plate $\o = (-5,5)\times(-2,2)$. 
The symbol next to the energy values for {$Z=-I_2$}
indicates whether the equilibrium shape is a cylinder (y) or not (n) 
after the stopping test \eqref{e:stop_criteria_num} is met.
Typical equilibria are displayed in Figure \ref{f:equilibrium} (left
and middle, the latter corresponding to a local discrete minimizer).
The numerical experiments indicate that the cylindrical shape (absolute
minimizer) is reached for {$Z=-I_2$} when $\tau$ and $h$ satisfy $\tau \leq C_0 h$ 
for a sufficiently small constant $C_0$. \cpam{This experimental condition
is more restrictive than the theoretically derived condition for 
mere convergence of our numerical scheme.
The symbol NoC for {$Z=-5I_2$} indicates that the sub-iterations of 
Algorithm~\ref{alg:fixed_point}} did not converge. The cylindrical shape is never
reached for {$Z=-5I_2$}; 
see Figure~\ref{f:equilibrium} (right) for a typical equilibrium configuration.}
 \label{t:clamped_energy_1}
\end{table}

\begin{figure}[ht!]
\centerline{\includegraphics[width=0.75\textwidth]{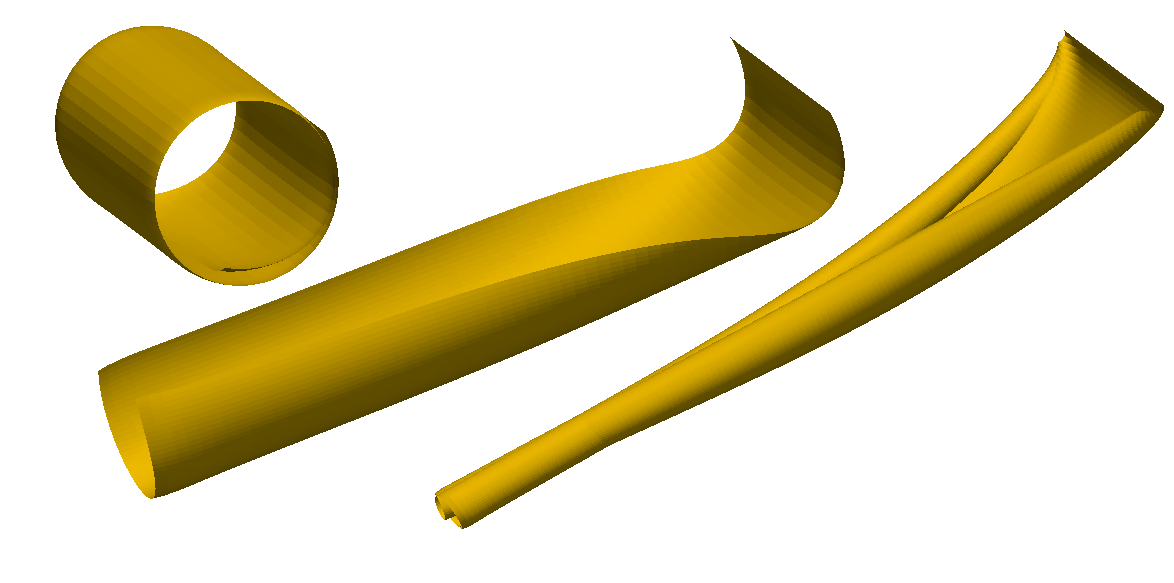}}
\caption{\small
Equilibrium configurations of clamped rectangle plates 
$\o = (-5,5)\times(-2,2)$ described in Section~\ref{ss:benchmark}
for different spontaneous curvatures and numerical parameters.
Left: cylinder shape when \cpam{$Z=-I_{2}$,} mesh refinement 6 and
$\tau=0.0025$; Middle: local minimum when \cpam{$Z=-I_{2}$,} mesh refinement 6
and $\tau=0.005$; Right: local minimizer when \cpam{$Z=-5I_{2}$,} mesh
refinement 7 and $\tau=0.01$.
If the spontaneous curvature is $1$ or smaller, then the cylinder
(absolute minimizer) is 
reached when the pseudo-timestep and the meshsize satisfy 
$\tau \leq C_0 h$ for a sufficiently small constant $C_0$. 
For relatively large spontaneous curvatures, the curvature in the
direction of the clamped side prevents the plate from bending
completely in the orthogonal direction, thereby creating a geometric
obstruction and leading to a (discrete) local minimizer for all
numerical \cpam{parameters} tried.}\label{f:equilibrium}
\end{figure}

\subsection{Asymptotics}\label{S:asymp}

In this section, we illustrate the predicted convergence rate of $O(\tau)$ 
\cpam{for the violation of the isometry constraint} in \eqref{isometry-defect}.
We consider again the clamped plate $\omega=(-5,5)\times(-2,2)$ described in 
Section~\ref{ss:benchmark} with a spontaneous curvature {$Z=-I_2$}.
The space discretizations are subordinate to $4,5,6$ and $7$ uniform
refinements of the initial \cpam{partition of the plate consisting of 1 rectangle}
 and are referred to as mesh $\#4,\#5,\#6$ and $\#7$, respectively.
The sub-iterations stopping tolerance of Algorithm
\ref{alg:fixed_point} is $\delta_{\text{stop}}=10^{-3}$.
The equilibrium isometry defect is defined to be
$$
ID_h(\by_h^\infty):= \frac{\| \lbrack \nabla_h \by_h^\infty \rbrack^\transp  \lbrack \nabla_h \by_h^\infty \rbrack - I_2\|_{L_h^1(\omega)}}{|\omega|}.
$$ 
The results
reported in Table \ref{t:clamped_isometry_1} indicate that the
predicted rate of convergence $O(\tau)$ in Proposition
\ref{prop:grad_flow} is recovered numerically.
This happens for both decreasing time steps $\tau$ on a fixed mesh
(columns 1, 2, and 3) as well as simultaneous reduction of $h$ and $\tau$
while keeping $h\sim \tau$ (diagonal). We stress that, according to
the discussion of Section \ref{ss:relaxation}, not all equilibrium
configurations are cylinders but the experimentally
observed linear rate applies to all of
them. This is consistent with Proposition \ref{prop:grad_flow} which
is not specific to absolute minimizers.

\begin{table}[ht!]
\begin{tabular}{c||cccc}
\diaghead{\theadfont Refinement}{$\tau$}{Mesh} &  \#4  & \#5 & \#6 & \#7 \\
\hline\hline
0.02 & 0.0559 &0.0488 & N/A &  N/A  \\
0.01 &   0.0327 & 0.0283 & 0.0249 &  0.0276 \\
0.005 & 0.0180 & 0.0155 & 0.0139 & N/A \\
0.0025 & 0.0094  & 0.0081 & 0.0077 & 0.0083 
\end{tabular}
\medskip
\caption{\small
Isometry defect $ID_h(\by_h^\infty)$ for the clamped plate $\o=(-5,5)\times(-2,2)$ at equilibrium  with spontaneous curvature {$Z=-I_2$}.
A sequence of time steps $\tau=0.02 \times 2^{-i}$, $i=0,1,2,3$, is
considered for different space resolutions.
A decay rate $O(\tau)$ is observed when the space discretization
remains unchanged (columns 1, 2 and 3) as well as when the meshsize
is reduced to satisfy $h\sim \tau$ (diagonal). Notice that the
isometry defect for a fixed time step is little affected by the space
resolution (rows 2 and 4).} \label{t:clamped_isometry_1}
\end{table}

\cpam{For an experimental convergence test for the deformations
we choose  $\o = (0,2\pi)^2$, define
\[
Z = -\left( 
\begin{array}{cc}
1 & 0 \\
0 & .5 
\end{array}
\right),
\] 
and consider clamped boundary conditions on the left side
$\p_D \o = \{x=0\}\times [0,2\pi]$. An exact solution is
given by 
\[
\by (x,y) = \big(\sin(x), y, 1- \cos(x)\big).
\]
Table~\ref{table:exp_conv_text} shows the scaled $L^2$ and $H^1$ errors
for the stationary configurations computed with 
Algorithms~\ref{alg:grad_flow_discr} and~\ref{alg:fixed_point};
norms were computed with a one-point Gaussian quadrature rule
on every element. 
The stopping criterion for the fixed-point iteration is
$\d_{\rm stop} = 10^{-4}$. The underlying meshes correspond to $\ell=3,4,5,6$ 
refinements of our coarse mesh so that $h_\ell = 2\pi 2^{-\ell}$.
We choose a step size proportional to the mesh size, i.e., 
we set $\tau_\ell = 2^{-\ell}/25$ for $\ell=3,4,5,6$. 
The obtained errors indicate a nearly linear experimental 
convergence rate.}


\begin{table}[ht!]
\begin{tabular}{r||cccc}
\diaghead{\theadfont Refinement}{Error}{Mesh} &  \#3  & \#4 & \#5 & \#6  \\
\hline\hline
$\|\be_h\|/|\o|^{1/2}$ &  0.8250 &  0.4273 &   0.2310 & 0.1220 \\
$\|\nabla \be_h\|/|\o|^{1/2}$ &  0.6723 &  0.3622 &  0.2002 & 0.1077 
\end{tabular}
\medskip
\caption{\small
\cpam{Scaled approximation errors $\be_h =\by-\by_h^\infty$
in an experimental convergence test 
with exact solution $\by$ given by a cylinder of radius~1. 
The approximations $\by_h^\infty$ are obtained with 
Algorithms~\ref{alg:grad_flow_discr} and~\ref{alg:fixed_point}
from a flat initial configuration and timestep sizes proportional
to meshsizes. The numbers indicate a nearly linear experimental
convergence rate.}} \label{table:exp_conv_text}
\end{table}

 \subsection{Effect of Aspect Ratio and Spontaneous Curvature}\label{ss:folding}

This section investigates numerically the influence of 
(i) the spontaneous curvature (i.e. difference in material properties
between the two plates) and (ii) the plates aspect ratio.
The numerical parameters in Algorithm \ref{alg:fixed_point}
are $\tau=0.005$, \cpam{$\delta_{\text{stop}}=10^{-4}$,} and the partition
corresponds to $5$ uniform refinements of the initial plate.

We consider plates $\omega := (-L,L) \times (-2,2)$ with different
lengths $L>0$ and define the aspect ratio to be $\rho:=L/2$.
The plates are clamped on the side $\partial_D\o =\{x=-L\} \times [-2,2]$.
It turns out that the tendency to bend in the clamped direction 
(accentuated for relatively large spontaneous curvatures in the
clamped direction according to Section \ref{S:asymp}) 
is attenuated for small aspect ratios.
We illustrate this in Figure~\ref{f:clamped_R}, which displays almost
equilibrium configurations for $\rho=5/2,3/2,1,1/2$ and spontaneous 
curvatures {$Z=-rI_2$} with $r=1,\cpam{3},5$. 
For large spontaneous curvatures, \cpam{Algorithm~\ref{alg:grad_flow_discr}}
did not always {reach} geometric equilibrium before the stopping test
\eqref{e:stop_criteria_num} was met. In addition, some 
pseudo-evolutions lead to severe folding and exhibit self-intersections, in
which case they are no longer representative from the physics standpoint.

\begin{figure}[ht!]
\centerline{\includegraphics[width=0.75\textwidth]{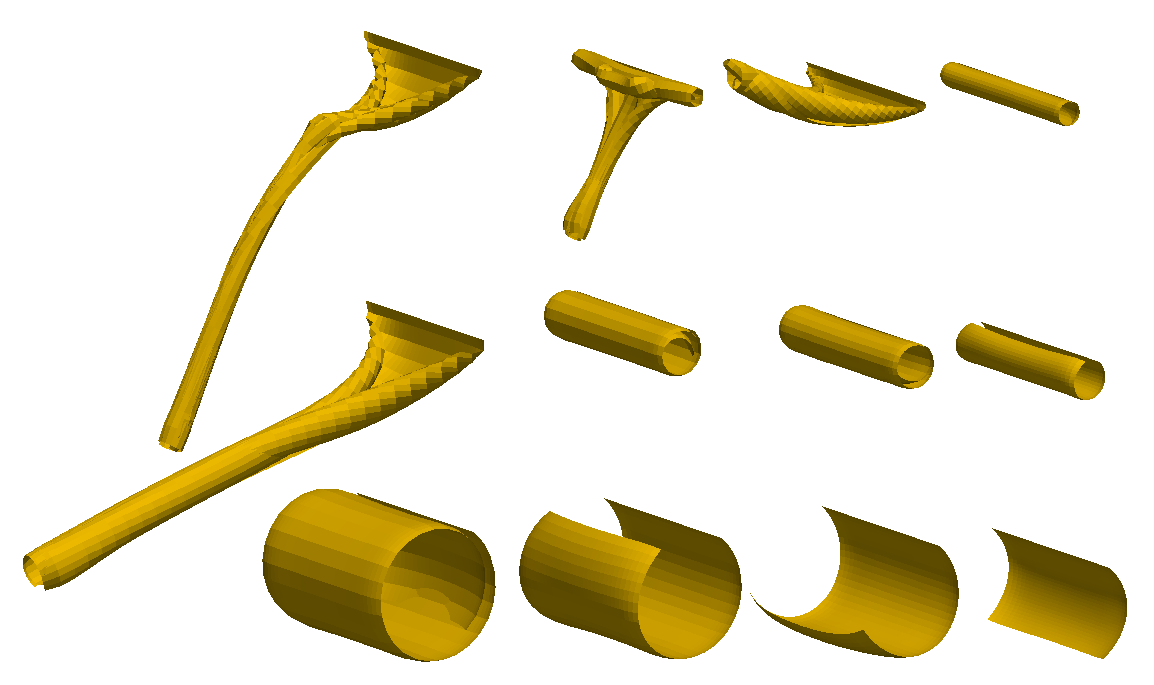}}
\begin{picture}(0,0)
\put(140,10){\small 3.3(4)}
\put(80,10){\small 6.5(8)}
\put(10,10){\small 9.8(12)}
\put(-60,10){\small 16.55(20)}
\put(140,80){\small 33.6(18)}
\put(80,80){\small 67.9(36)}
\put(10,80){\small 103.8(108)}
\put(-60,80){\small 188.6(180)}
\put(140,160){\small 96.7(50)}
\put(80,160){\small 220.4(100)}
\put(10,160){\small 305.3(300)}
\put(-60,160){\small 519.6(500)}
\end{picture}
\caption{\small
Equilibrium shapes of bilayer plates for several aspect-ratios $\rho$
(from left to right $\rho=5/2,3/2,1,1/2$) and spontaneous curvatures
{$Z=-rI_2$} (from top to bottom $r=5,3,1$).
Decreasing the aspect ratio restores the ability for the plate to fold 
into a cylindrical shape for larger spontaneous curvatures.
For instance, this is the case for plates with parameters $r=3$ and 
$\rho=3/2$ or $r=5$ and $\rho =1/2$.
\cpam{Notice, however, that small regions around the free corners have not completely relaxed to equilibrium. 
This effect is due to the violation of the isometry constraint and reduced upon decreasing the discretization parameters as well as the stopping criteria.}
The numbers below each stationary configuration are the 
corresponding approximate energies. \cpam{For comparison, the energies of corresponding plates with principal curvatures of $\frac 1 r$ and $0$ are given between parenthesis.} 
}
\label{f:clamped_R}
\end{figure}
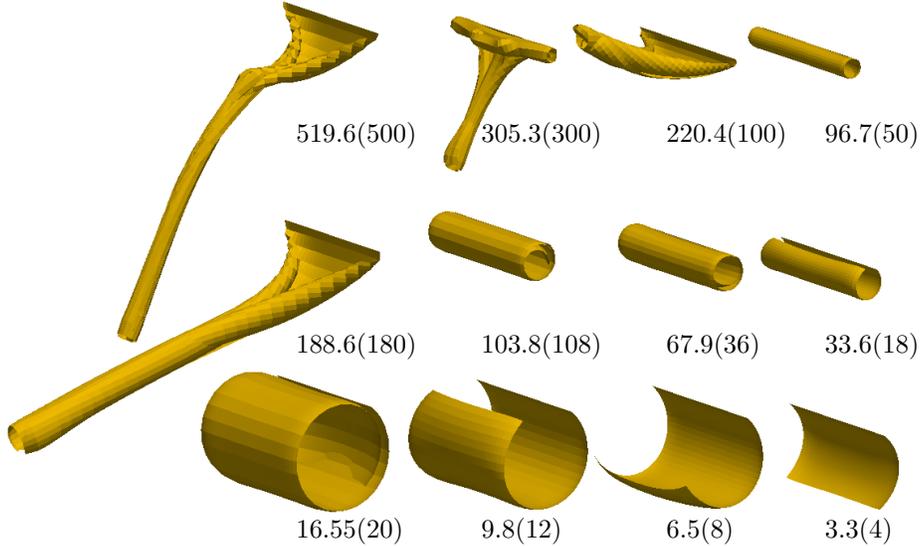

\subsection{Boundary Conditions and Shapes}\label{ss:boundary}

We consider now different boundary conditions and plate
shapes. We intend to examine the robustness of \cpam{our numerical scheme}
in different situations and investigate plate shapes which are not studied
in \cite{Schm:07a} and for which we do not know the absolute
minimizers.

\medskip
{\it Boundary conditions:} 
We start with the plate $\omega = (-3,3) \times (-2,2)$ clamped in a
neighborhood of the bottom left corner, namely
{$\partial_D\o=\lbrace x=-3 \rbrace \times (-2,0) \cup (-3,0) \times \lbrace y=-2 \rbrace$}.
We impose the spontaneous curvature {$Z=-I_2$}, choose the numerical
parameters $\tau=0.005$,  $\delta_{\text{stop}}=10^{-4}$ in Algorithm
\ref{alg:fixed_point}, and use a partition of $\o$ with 5 uniform refinements.
Several intermediate shapes of the discrete
gradient flow are depicted in Figure~\ref{f:cornerclamped}. \cpam{The
equilibrium configuration consists of a flat and a cylindrical part separated
by a free boundary that connects points on the boundary at which boundary 
conditions change.}

\begin{figure}[ht!]
\centerline{\includegraphics[width=0.3\textwidth]{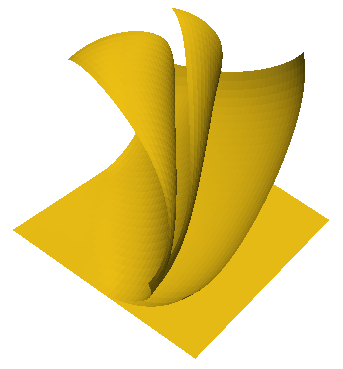}}
\begin{picture}(0,0)
\thicklines
\put(110,110){\qbezier(0,0)(-10,40)(-50,50)}
\put(55,163){\vector(-2, 1){0}}
\end{picture}
\vskip-0.8cm
\caption{\small
Different snapshots of the deformed corner-clamped plate with
spontaneous curvature {$Z=-I_2$}.
The equilibrium shape has energy $11.616$ and is not a cylinder.
It is worth comparing with the side-clamped plate discussed in 
Section~\ref{ss:folding}, which leads to a cylinder of smaller 
approximate energy $9.81$.} \label{f:cornerclamped}
\end{figure}
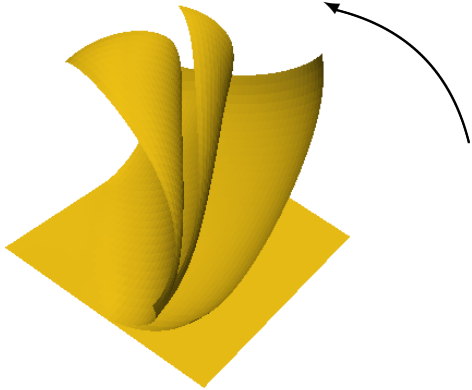

\medskip
{\it Shapes:}
We now consider the I-shaped and the O-shaped plates depicted in Figure~\ref{f:IO}.
The finite element meshes contain $7168$ and $8192$ quasi-uniform rectangles respectively.
We set {$Z=-5I_2$}, $\tau=0.005$ and $\delta_{\text{stop}}=10^{-3}$. 
Relaxation toward numerical equilibrium shapes, which are not
cylinders, is depicted for both plates in Figure \ref{f:IOeq}. We
stress that for the I-shaped plate the curvature in the clamped direction
dominates the other, thereby leading to a cigar shape that opens up at
the bottom to accomodate the boundary condition. In contrast, the
O-shaped plate is more rigid to bending and develops {\it dog-ears} at
the free corners which prevent further bending.

\begin{figure}[ht!]
\centerline{\hfill \includegraphics[width=0.3\textwidth]{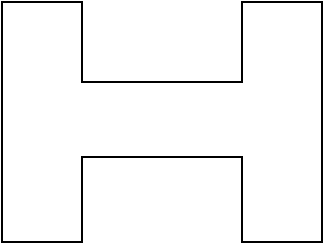}\hfill \includegraphics[width=0.3\textwidth]{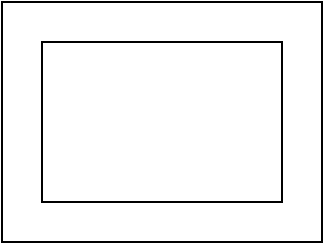}\hfill}
\begin{picture}(0,0)
\put(100,0){$10$}
\put(100,40){$\frac {20} 3$}
\put(55,60){$\frac {8} 3$}
\put(20,60){$4$}
\put(-180,60){$4$}
\put(-157,0){$\frac 52$}
\put(-130,30){$\frac 34$}
\put(-105,55){$5$}
\end{picture}
\caption{\small
Geometry of the I-shaped and O-shaped plates: the numbers indicate the length of the sides.
In both cases, the clamped edge is the far left vertical edge.}\label{f:IO}
\end{figure}
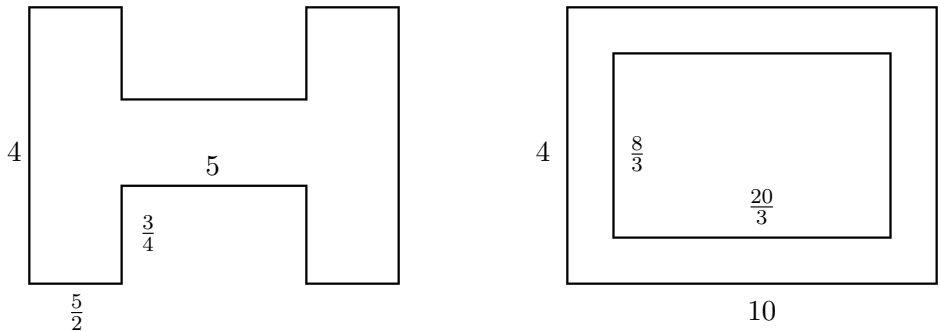
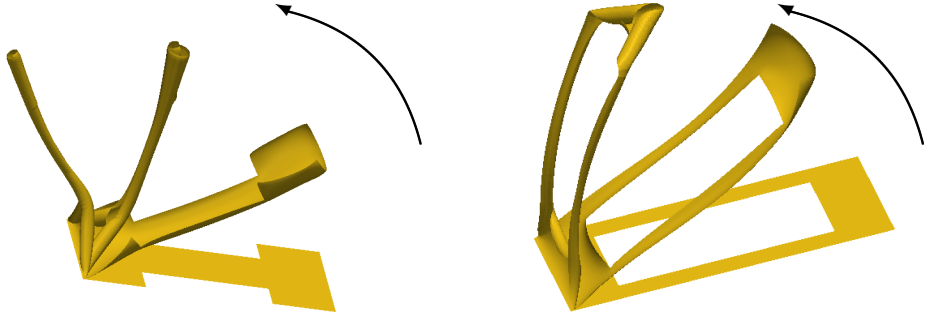
\begin{figure}[ht!]
\centerline{\hfill \includegraphics[width=0.3\textwidth]{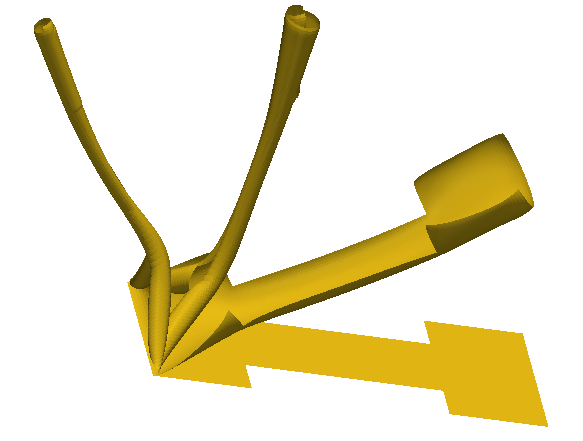} \hfill
 \includegraphics[width=0.3\textwidth]{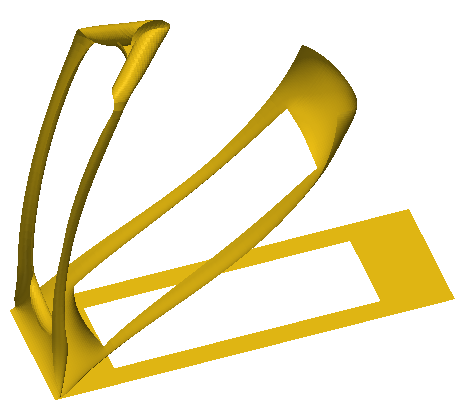}
\hfill
}
\begin{picture}(0,0)
\thicklines
\put(180,80){\qbezier(0,0)(-10,40)(-50,50)}
\put(125,133){\vector(-2, 1){0}}
\put(-10,80){\qbezier(0,0)(-10,40)(-50,50)}
\put(-65,133){\vector(-2, 1){0}}
\end{picture}
\vskip-0.5cm
\caption{\small
Different snapshots of the deformed I-shaped plate (left) and
  O-shaped plate (right).
{The corresponding stationary energies with spontaneous curvatures $Z=-5
I_2$ are 404.57 and 314.152 respectively or about $12.448$ and
$14.137$ relative to the plate areas.
For comparison, the numerical stationary energy of a plate
$\omega=(-5,5) \times (-2,2)$ under the same boundary condition and
spontaneous curvature is $519.6$ or $12.99$ once divided by the plate area  (see Figure~\ref{f:clamped_R}).
It turns out that compared to the full plate,  the stationary
numerical energy per unit area 
is smaller for the I-shaped plate and greater for the O-shaped plate.}} \label{f:IOeq}
\end{figure}

\subsection{Anisotropic Spontaneous Curvature}
We now turn our attention to anisotropic spontaneous curvatures, namely
to matrices {$Z$} with different eigenvalues. In the first two examples
the eigenvectors are aligned with the coordinate axis, but not in the
third example. The spontaneous curvature is given by either
\begin{equation}\label{e:anisotropic}
{Z=\left( 
\begin{array}{cc}
-5 & 0 \\
0 & -a 
\end{array}
\right),
\qquad
Z=\left( 
\begin{array}{cc}
-3 & 2 \\
2 & -3 
\end{array}
\right),}
\end{equation}
with $a=1$ or $-5$. The plate  is {$\o=(-2,2)\times(-3,3)$} and the
numerical parameters in Algorithm
\ref{alg:fixed_point} are $\tau=0.005, \delta_{\text{stop}}=10^{-3}$.

\medskip
{\it Dominant curvature:} 
With $a=1$ being the curvature in the clamped direction, we notice
a rather minimal bending effect in such a direction.
The plate pseudo-evolutions are displayed in Figure \ref{f:anise-1} 
which shows an almost perfect rolling to a cylinder {of energy $42.09$}.
\begin{figure}[ht!]
\centerline{\includegraphics[width=0.5\textwidth]{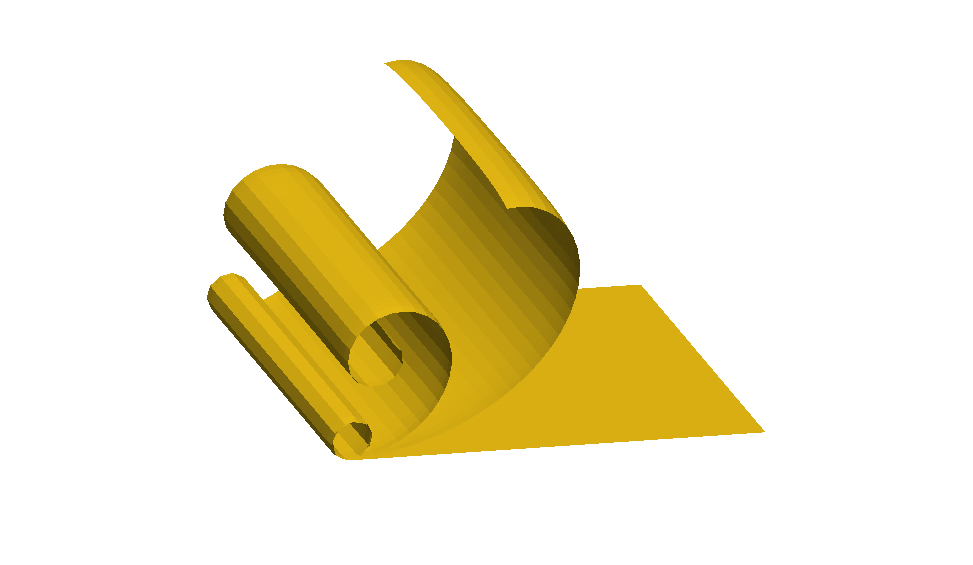}}
\begin{picture}(0,0)
\thicklines
\put(80,50){\qbezier(0,30)(-10,60)(-30,60)}
\put(45,110){\vector(-2, 0){0}}
\end{picture}
\vskip-0.5cm
\caption{\small
Deformation of a plate with anisotropic curvature given by
\eqref{e:anisotropic} with $a=1$. The spontaneous curvature is $1$ in
the clamped direction, its effect being barely noticeable, whereas it
is $5$ in the orthogonal direction.
The equilibrium shape is a cylinder (absolute minimizer) {with an energy of $42.09$}. Compare with
the case $a=5$, presented in {Section~\ref{ss:folding}}, for which
the cylindrical shape is not achieved before the stopping test 
\eqref{e:stop_criteria_num} is met.} \label{f:anise-1}
\end{figure}
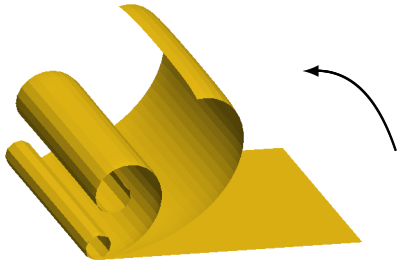
\begin{figure}[ht!]
\centerline{\includegraphics[width=0.7\textwidth]{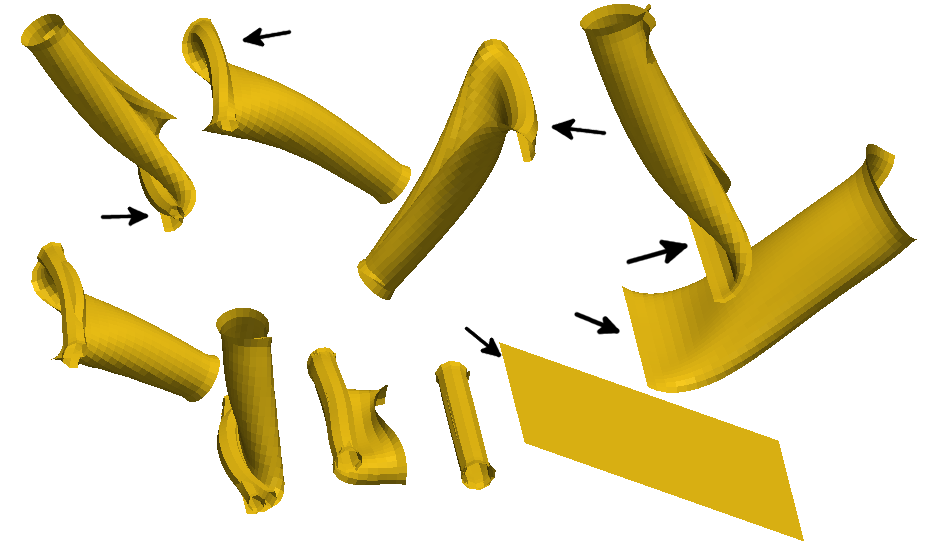}}
\caption{\small
Deformation of a plate with anisotropic curvature given by
  \eqref{e:anisotropic} with $a=-5$. The spontaneous curvatures are $-5$ in the
  clamped direction and $5$ in the perpendicular direction, which
  eventually dominates the former and leads to a cylindical shape
  after three full rotations. Snapshots are displayed
  counterclockwise, starting from the bottom right, 
 {for 0.0, 0.3, 2.0, 4.0, 6.0, 10.0, 15.0, 18.0, 25.0, 172.0 times $10^3$
  iterations}  of Algorithm \ref{alg:grad_flow_discr}. 
  The arrows indicate the clamped side.} \label{f:anise-neg}
\end{figure}

{\it Curvatures with opposite signs:}
We take $a=-5$ to be the curvature in the clamped direction. This choice
models the tendency of the plate to bend equally in opposite directions along
the coordinate axes (principal directions). This is noticeable in the
second and third snapshots in Figure~\ref{f:anise-neg},
the latter also exhibiting self-crossing of the free corners.
After three complete rotations, the plate relaxes to a cylindrical
shape (absolute minimizer). 
\cpam{Surprisingly, a cylindrical shape is reached before
the stopping test \eqref{e:stop_criteria_num} is met, unlike 
the case $a=5$ (see first row - second column in Figure~\ref{f:clamped_R}).}

\medskip
{\it Corkscrew shape:}
We consider now the second anisotropic spontaneous curvature {$Z$} in 
\eqref{e:anisotropic}, which has eigenpairs
\[
\mu_1=5\quad\be_1=[1,-1]^T,\qquad
\mu_2=1\quad\be_2=[1,1]^T.
\]
This means that we still have principal curvatures $5$ and $1$ but
with principal directions $\be_1$ and $\be_2$ forming the angle
$\pi/4$ with the coordinate axes.
The deformation of this plate towards its equilibrium shape is
displayed in Figure~\ref{f:aniserot}.
\cpam{The plate exhibits a {\it corkscrew} shape before self-intersecting 
and continuing its deformation to a conic shape.
In fact, a cylindrical shape is not reached before the stationarity test 
\eqref{e:stop_criteria_num} is met.
We emphasize that this is not in contradiction with \cite{Schm:07a}, where scalar spontaneous curvatures are considered, and shed some light on equilibrium configurations when the two principal spontaneous curvature directions are not parallel and orthogonal to the clamped side. 
Notice, however, that the equilibrium energy obtained is {$61.31$}, which is larger than the cylindrical shape obtained when the principal direction aligned with the coordinate axes (see Figure~\ref{f:anise-1})}.

\begin{figure}[ht!]
\begin{center}
\begin{tabular}{ccc}
\includegraphics[width=0.25\textwidth]{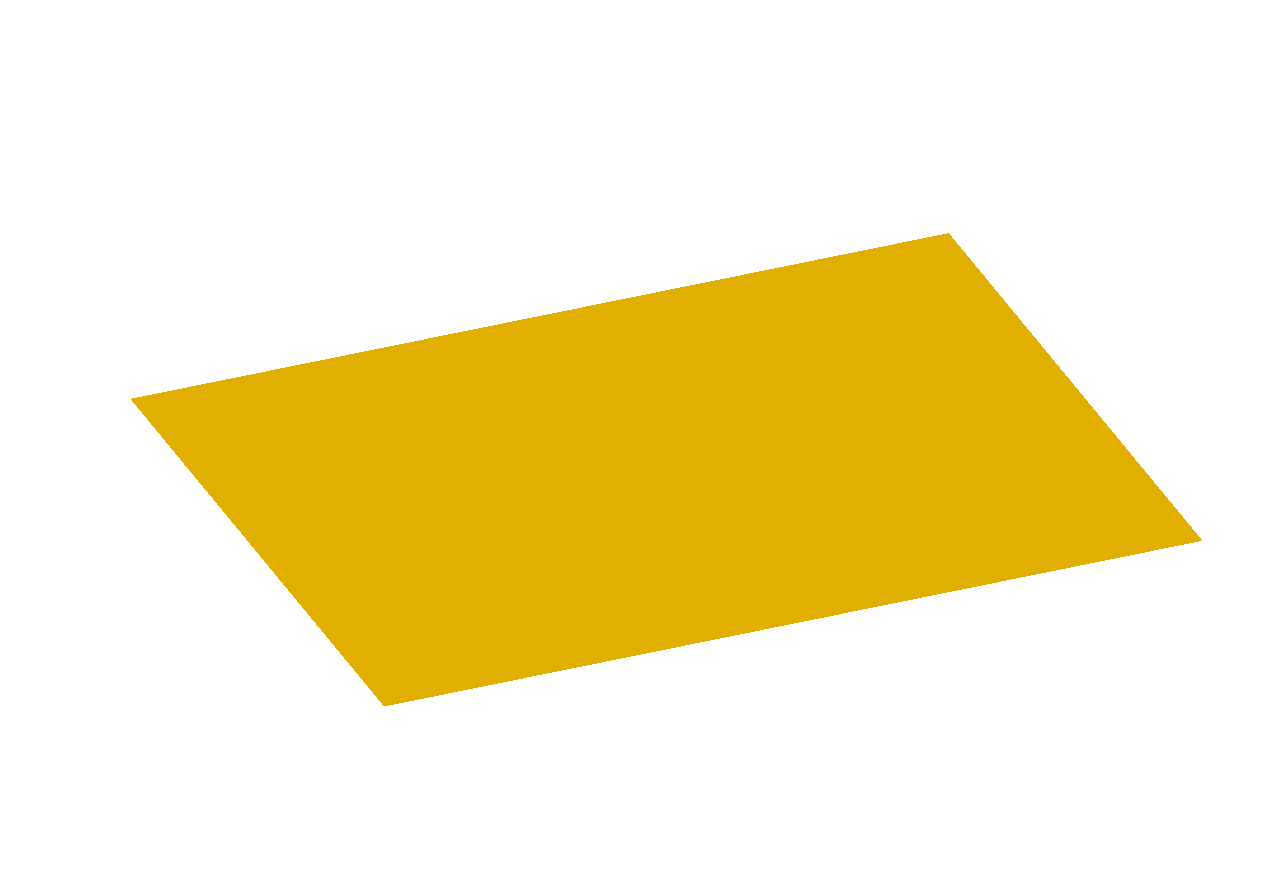}   &  \includegraphics[width=0.25\textwidth]{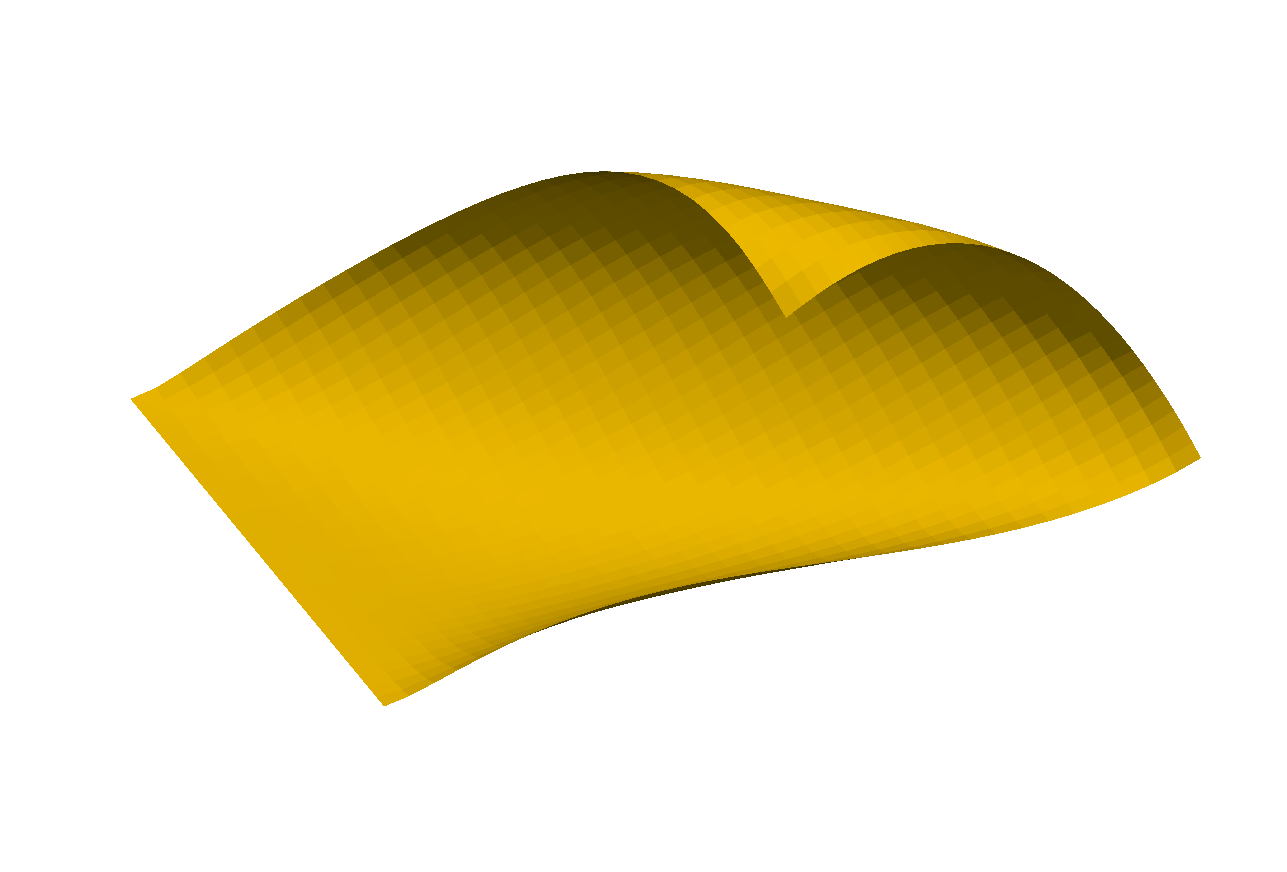}
&   \includegraphics[width=0.25\textwidth]{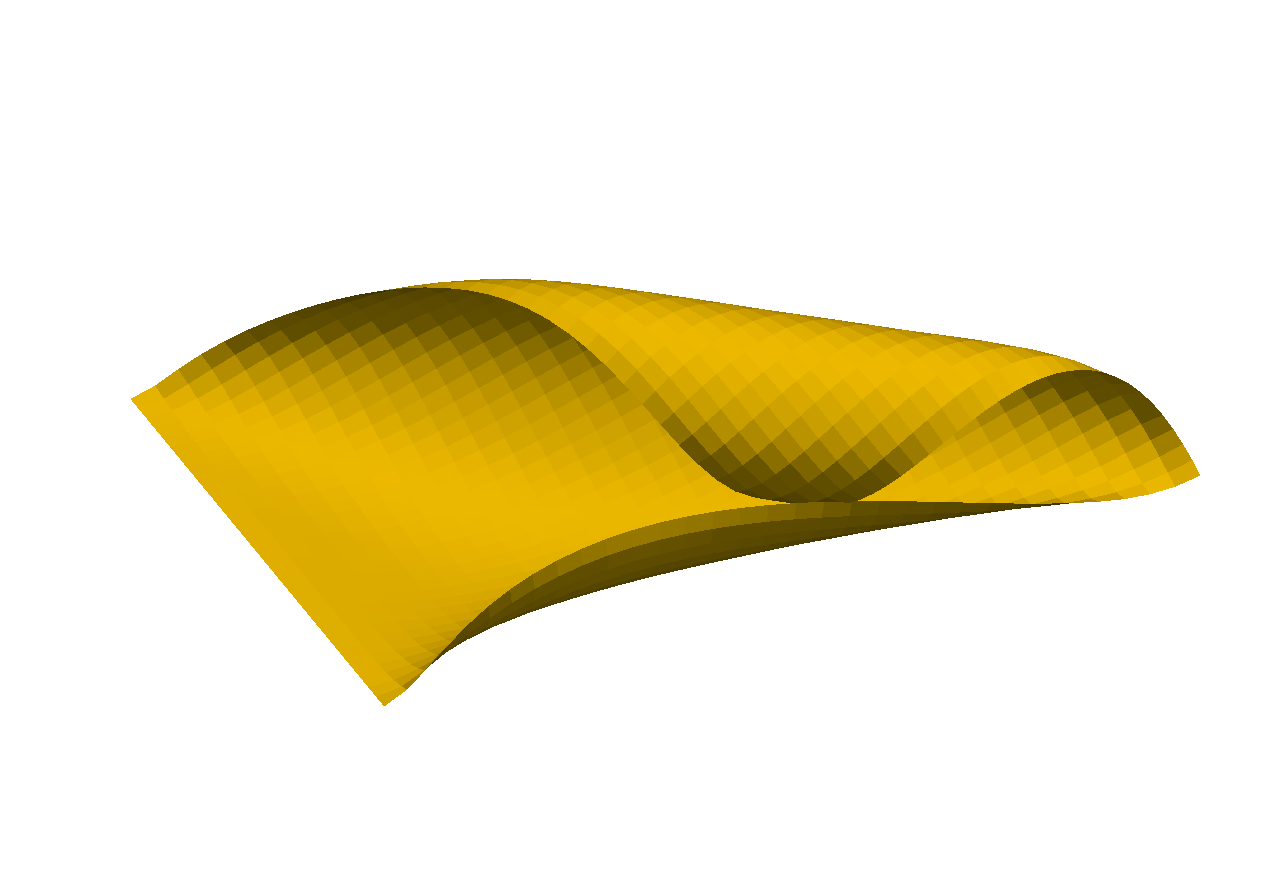} \\ 
\includegraphics[width=0.25\textwidth]{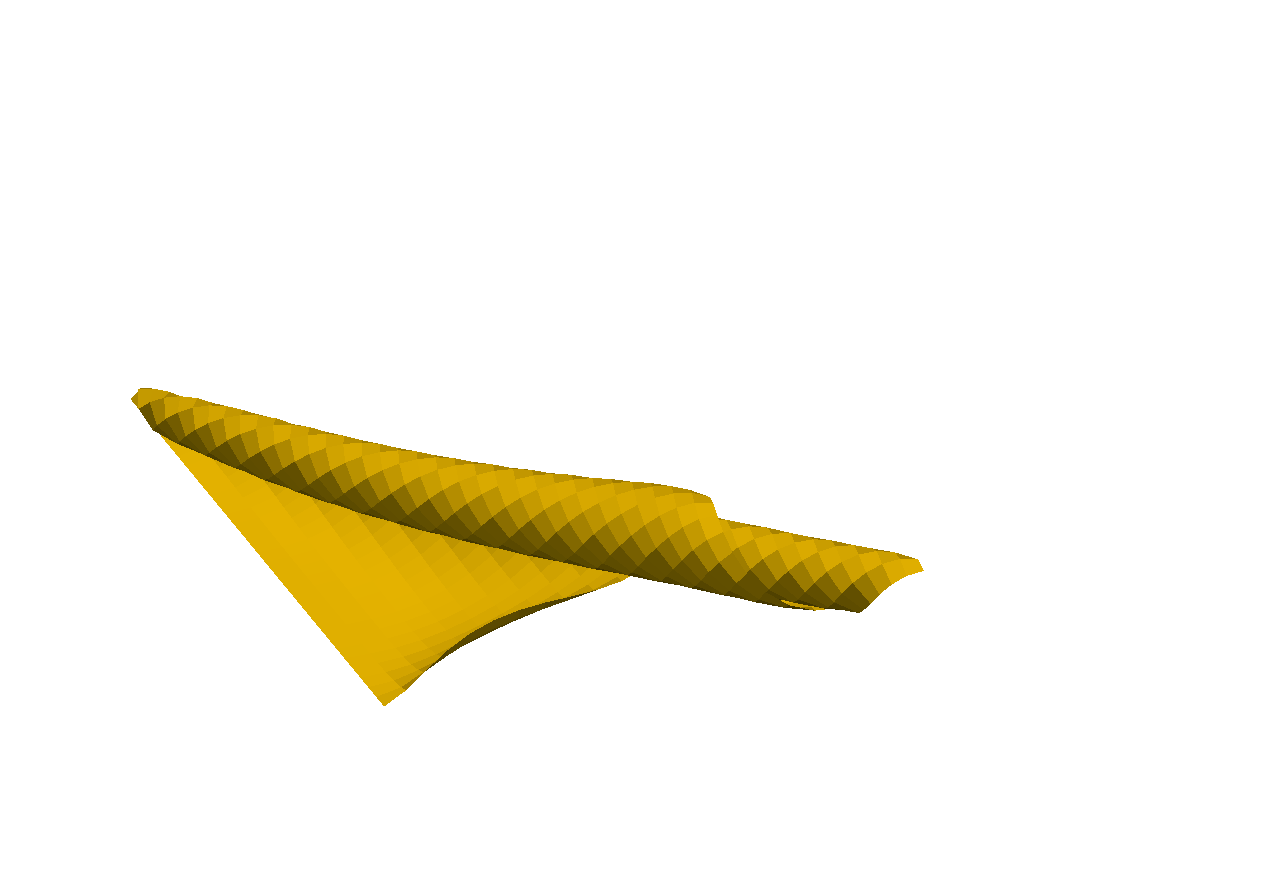}  &  \includegraphics[width=0.25\textwidth]{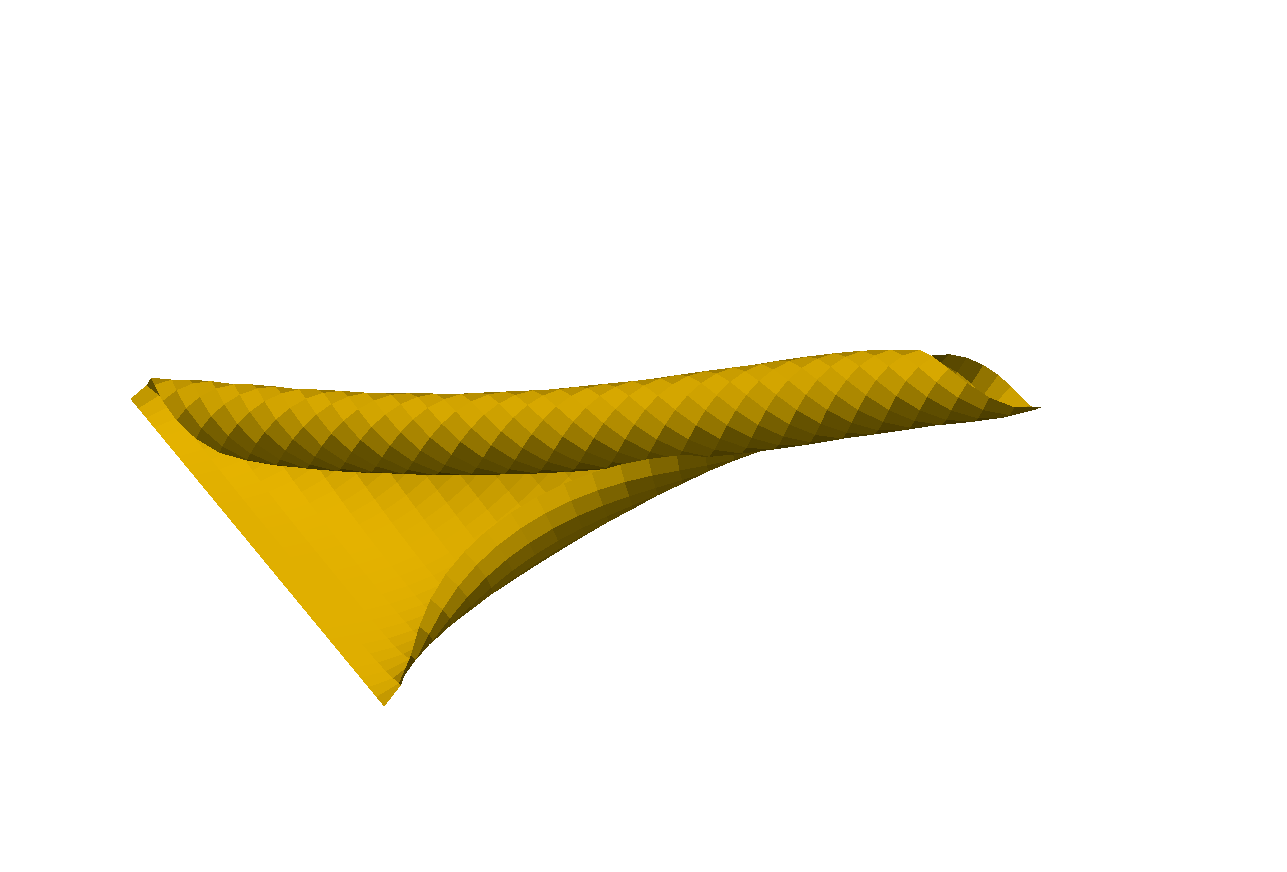}&   
 \includegraphics[width=0.25\textwidth]{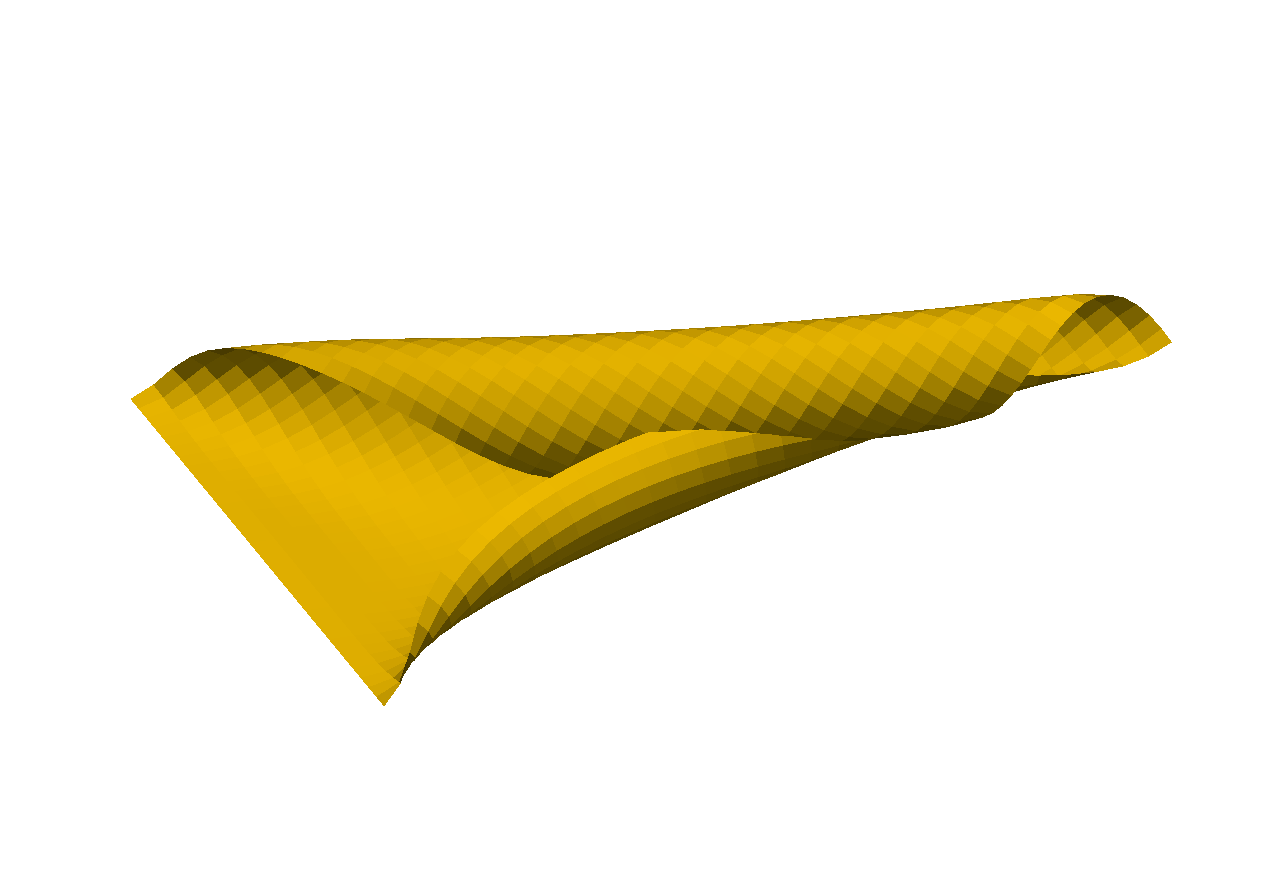} 
\end{tabular}
\end{center}
\begin{picture}(0,0)
\thicklines
\put(-80,160){\vector(2,0){30}}
\put(50,160){\vector(2,0){30}}%
\put(130,110){\vector(0,-2){30}}%
\put(-60,30){\vector(-2,0){30}}
\put(70,30){\vector(-2,0){30}}%
\end{picture}
\caption{\small
{Deformation of a plate with the second anisotropic curvature {$Z$}
in \eqref{e:anisotropic}. The principal curvatures are $5$ and $1$ but
the principal directions form an angle $\pi/4$ with the coordinate axes.
The snapshots are displayed clockwise starting at the top
left, for {0.0, 0.1, 0.4, 1.9, 3.0, 324.0 times $10^3$ iterations} of
Algorithm \ref{alg:grad_flow_discr}.
The plate adopts a corkscrew shape before self-intersecting.} }
\label{f:aniserot} 
\end{figure}
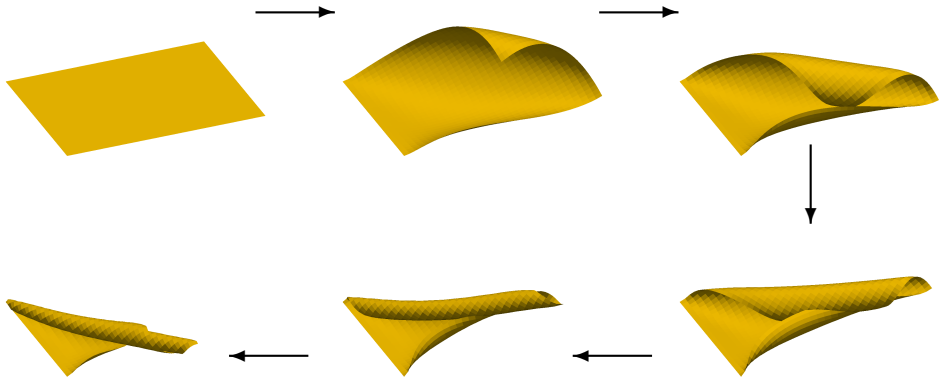

\subsection{Energy Decay and Time Scales}

One critical aspect missing in this study
is the design of (pseudo)-time adaptive algorithms 
able to cope with the disparate time scales inherent to the 
underlying energies.
To illustrate this point, we plot in Figure~\ref{f:energy-decay}
the energy decay of the clamped plate $\o$ of
Section~\ref{ss:benchmark} for spontaneous curvatures {$Z=-I_2$ and $Z=-5I_2$}.
Both energies exhibit a rapid decay at the very beginning of 
the deformation and very slow decay at the end. 
The numerical parameters of Algorithm \ref{alg:fixed_point}
used for these simulations are 
$\tau=0.005$, $\delta_{\text{stop}}=10^{-4}$ and the finite element
partition corresponds to 5 uniform refinements of the original plate.

\begin{figure}[ht!]
\centerline{\includegraphics[width=0.5\textwidth]{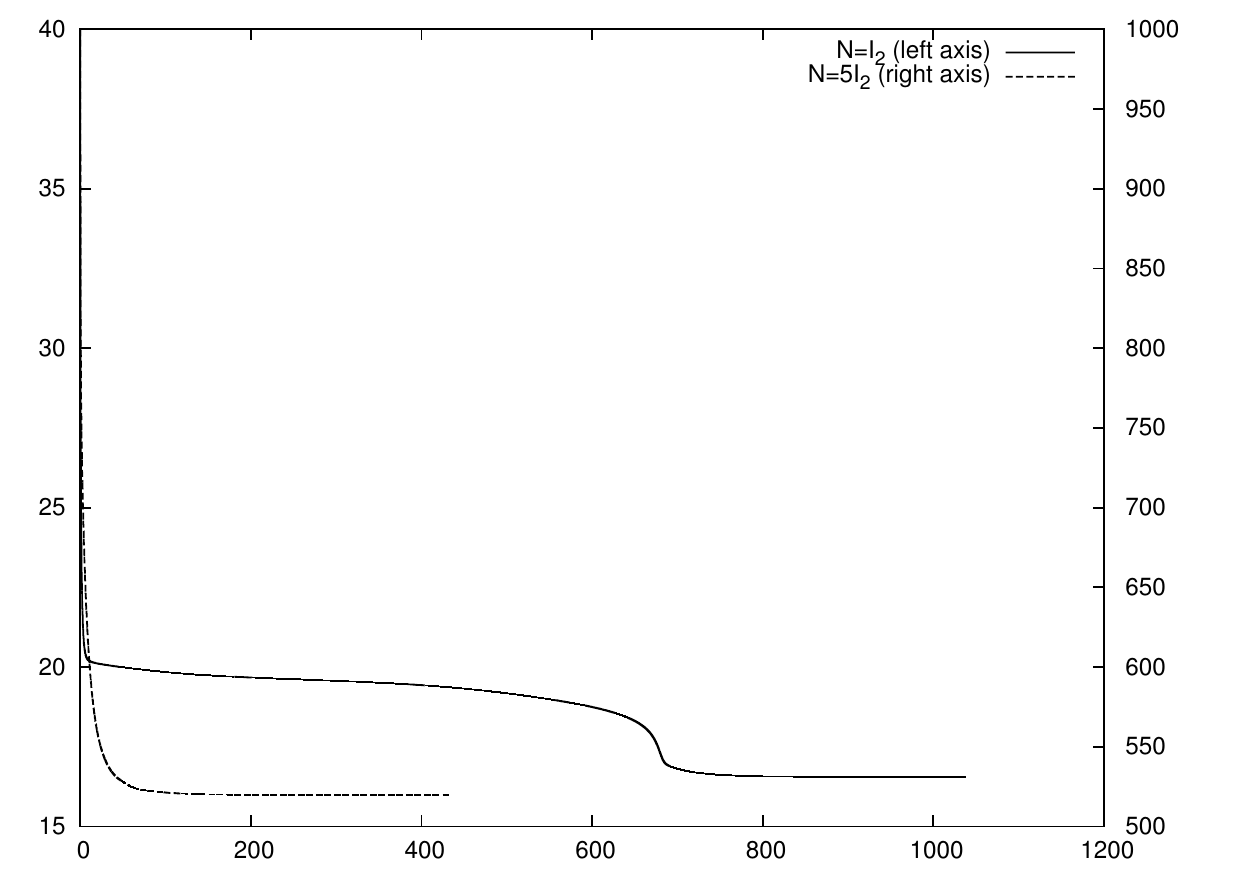}}
\caption{\small
Energy decay versus pseudo-time for {$Z=-rI_2$} with $r=1$ and $r=5$. 
The cylindrical shape is reached when $r=1$ after 207,000 pseudo-timesteps 
(total of 210.469 solves accounting for the sub iterations). 
When $r=5$, the equilibrium shape is reached faster after 86.600 pseudo-timesteps  (total of 129.682 solves accounting for the sub iterations).
However, the equilibrium reached is not a cylinder as already pointed out in Sections~\ref{ss:relaxation} and \ref{ss:folding}; see for instance Figure \ref{f:equilibrium}.
The energy decays fast at the very beginning of the relaxation process
in both cases. In addition when $r=1$, a second rapid decay arises with the unfolding
in the clamped direction; see iteration 130,000 in Figure
\ref{f:benchmark} (6th snapshot).}\label{f:energy-decay}
\end{figure}


\bigskip
{\bf Acknowledgements.}
We are indebted to S. Conti who suggested a reduced model leading to
that of Section \ref{S:model-reduction}. We are also thankful to
E. Smela who stimulated our curiosity to study folding patterns of bilayer plates
via several laboratory experiments and discussions.
Finally, we express our gratitude to W. Bangerth
for participating in several discussions regarding the implementation of the Kirchhoff quadrilaterals with deal.II \cite{BHK:07}.


\bibliographystyle{acm}

\end{sloppypar} 
\end{document}